\documentclass[11pt, a4paper]{amsart}
\usepackage{amsmath,ifthen}
\usepackage[hmargin=32mm, vmargin=27mm, includefoot, twoside]{geometry}
\usepackage[
bookmarksopen=true]{hyperref}
\usepackage[latin1]{inputenc}  
\usepackage[T1]{fontenc}       
\usepackage[english]{babel}
\usepackage{amssymb, enumerate}
\usepackage{latexsym}
\usepackage{mathrsfs}
\usepackage{xspace}
\usepackage{enumerate, paralist}
\usepackage{graphics, graphicx, epsfig}
\usepackage[all, cmtip]{xy}
\usepackage{verbatim} 

\usepackage[usenames,dvipsnames]{color}

\newcommand{\showcomments}{yes}

\newsavebox{\commentbox}
{\ifthenelse{\equal{\showcomments}{yes}}%
{\footnotemark
    \begin{lrbox}{\commentbox}
    \begin{minipage}[t]{1.00in}\raggedright\sffamily\tiny
    \footnotemark[\arabic{footnote}]}
{\begin{lrbox}{\commentbox}}}%
{\ifthenelse{\equal{\showcomments}{yes}}%
{\end{minipage}\end{lrbox}\marginpar{\usebox{\commentbox}}}
{\end{lrbox}}}

\newtheorem{prop}{Proposition}[section]
\newtheorem{thm}[prop]{Theorem}
\newtheorem*{thm*}{Theorem}

\newtheorem*{addendum*}{Addendum}
\newtheorem{cor}[prop]{Corollary}
\newtheorem{lem}[prop]{Lemma}
\newtheorem{conj}[prop]{Conjecture}
\newtheorem{thmintro}{Theorem}

\newtheorem*{conj*}{Conjecture}

\newtheorem{corintro}[thmintro]{Corollary}

\newtheorem*{claim*}{Claim}
\newtheorem*{convention*}{Convention}
\theoremstyle{definition}
\newtheorem*{defn*}{Definition}

\newtheorem{remark}[prop]{Remark}

\newtheorem*{scholium*}{Scholium}
\theoremstyle{remark}
\newtheorem{example}[prop]{Example}
\newtheorem*{example*}{Example}
\numberwithin{equation}{section}

\newcommand{\FF}{\mathbf{F}}

\newcommand{\NN}{\mathbf{N}}

\newcommand{\QQ}{\mathbf{Q}}

\newcommand{\ZZ}{\mathbf{Z}}

\newcommand{\GL}{\mathrm{GL}}
\newcommand{\SL}{\mathrm{SL}}
\newcommand{\PSL}{\mathrm{PSL}}
\newcommand{\PGL}{\mathrm{PGL}}

\newcommand{\la}{\langle}
\newcommand{\ra}{\rangle}
\newcommand{\inv}{^{-1}}

\newcommand{\MM}{\mathbb{M}}

\newcommand{\centra}{\mathscr{Z}}

\DeclareMathOperator{\chr}{char} 
 
 \DeclareMathOperator{\Fix}{Fix} 
 \DeclareMathOperator{\Ker}{Ker}
 \DeclareMathOperator{\End}{End}

\newcommand{\bd}{\partial} 

\def\Aut{\mathop{\mathrm{Aut}}\nolimits}

\def\Sym{\mathop{\mathrm{Sym}}\nolimits}

\def\max{\mathop{\mathrm{max}}\nolimits}

\begin{document}
\title[Trees, contractions and Moufang sets]{Trees, contraction groups and Moufang sets}
\author[P-E. Caprace]{Pierre-Emmanuel Caprace*}
\address{UCLouvain, Chemin du Cyclotron 2, 1348 Louvain-la-Neuve, Belgium}
\email{pierre-emmanuel.caprace@uclouvain.be}
\thanks{*F.R.S.-FNRS Research Associate, supported in part by FNRS grant F.4520.11 and by the European
Research Council (ERC)/ grant agreement 278469}
\author[T. De Medts]{Tom De Medts}
\address{Ghent University, Dept. of Mathematics,
\mbox{Krijgslaan} 281 -- S22, 9000 Gent, Belgium}
\email{tdemedts@cage.ugent.be}

\date{January 2012}

\begin{abstract}
We study closed subgroups $G$ of the automorphism group of a locally finite tree $T$ acting doubly transitively on the boundary.  We show that if the stabiliser of some end is metabelian, then there is a local field $k$ such that $\PSL_2(k) \leq G \leq \PGL_2(k)$. We also show that  the contraction group of some hyperbolic element is closed and torsion-free  if and only if  $G$ is (virtually) a rank one simple $p$-adic analytic group for some prime $p$. A key point is that if some contraction group is closed, then $G$ is boundary-Moufang, meaning that the boundary $\bd T$ is a Moufang set. We collect basic results on Moufang sets arising at infinity of locally finite trees, and provide a complete classification in case the root groups are torsion-free. 
\end{abstract}

\maketitle

\setcounter{tocdepth}{1}    
\tableofcontents

\section{Introduction}
Let $T$ be a locally finite tree without vertices of valency one and $\partial T$ denote its set of ends. The group $\Aut(T)$, endowed with the compact open topology, is a second countable totally disconnected locally compact group. This paper is devoted to the study of closed subgroups $G \leq \Aut(T)$ acting doubly transitively on $\bd T$. Examples of such groups arise from rank one simple algebraic groups over local fields via Bruhat--Tits theory, and are thus linear. On the other hand, many non-linear examples exist as well, e.g.\@ the full group $\Aut(T)$ when $T$ is semi-regular, or subgroups of $\Aut(T)$ with prescribed local actions as constructed by Burger--Mozes~\cite{BurgerMozes:trees1}, or certain complete Kac--Moody groups of rank two. It is thus a natural problem to find some criteria allowing one to distinguish between the linear subgroups of $\Aut(T)$ and the non-linear ones. This question is the starting point of this work. 

\subsection{Linearity criteria for boundary-transitive groups}

The following result provides a necessary and sufficient condition for a boundary-transitive group to essentially coincide with $\SL_2(k)$ over some local field $k$.

\begin{thmintro}\label{thm:MetabelianStab}
Let $G \leq \Aut(T)$ be a closed non-compact subgroup acting transitively on $\bd T$. Assume there is some $\xi \in \bd T$ such that the stabiliser $G_\xi$ is metabelian%
\footnote{A group is called {\em metabelian} if its commutator subgroup is abelian, or equivalently, if it is solvable of derived length at most $2$.}. 

Then there is a  non-Archimedean local field $k$ such that $\PSL_2(k) \leq G \leq \PGL_2(k)$ as closed subgroups, and $T$ is equivariantly isometric (up to scaling) to the Bruhat--Tits tree of $\PGL_2(k)$. 
\end{thmintro}

It is an interesting fact that a non-compact closed subgroup $G \leq \Aut(T)$ which acts transitively on $\bd T$ is automatically doubly transitive on $\bd T$ (see \cite[Lemma~3.1.1]{BurgerMozes:trees1}). In an appendix, we show that an abstract triply transitive permutation group with metabelian stabilisers must automatically be a projective group. This is however not sufficient to deduce Theorem~\ref{thm:MetabelianStab}, since $\PSL_2(k)$ is doubly transitive, but not necessarily  triply transitive, on the projective line over $k$ in general. 

\medskip
In order to present the next statement,  we first recall that,  given an element $\alpha$ in a topological group $G$,  the set
\[
U_\alpha = \{ g \in G \mid \lim_{n \to \infty} \alpha^n g \alpha^{-n} = 1\}
\]
is a subgroup of $G$ called the \textbf{contraction group} associated to $\alpha$. It should be emphasized that the contraction group of an element is not a closed subgroup in general, even if $G$ is locally compact. The simplest illustration of this possibility is given by the `unrestricted lamplighter group' $(\prod_\ZZ F) \rtimes \ZZ$, where $F$ is a non-trivial finite group and $\ZZ$ acts by shifting the indices. One verifies easily that the contraction group associated with the generator $\alpha$ of the cyclic factor $\ZZ$ is a non-closed dense subgroup of the compact normal subgroup $\prod_\ZZ F$. On the other hand, it is true that the contraction groups are always closed if $G$ belongs to some special classes of locally compact groups, e.g.\@ Lie groups (see \cite{HazodSiebert}) or $p$-adic analytic groups (see \cite{Wang84}). 
The following result provides a partial converse to this property among boundary-transitive automorphism groups of trees. 

\begin{thmintro}\label{thm:padic}
Let $G \leq \Aut(T)$ be a closed non-compact subgroup acting transitively on $\bd T$.  Assume that for some hyperbolic element $\alpha \in G$,  the contraction group $U_\alpha$ is closed and torsion-free.

Then there is a prime $p$, a $p$-adic field $k$, a semisimple algebraic $k$-group $\mathbf G$ of $k$-rank one, and a continuous open homomorphism $\mathbf G(k) \to G$ whose image $G^+$ is a characteristic subgroup of finite index.  Moreover $T$ is equivariantly isometric (up to scaling) to the Bruhat--Tits tree of $\mathbf G(k)$.  
\end{thmintro}

The proofs of both theorems above are based on a strong relation between boundary-transitive automorphism groups of trees and Moufang sets, which we shall now describe. 

\subsection{Relation to Moufang sets}


A subgroup $G \leq \Aut(T)$ is called \textbf{boundary-Moufang} if $G$ is closed and possesses a full conjugacy class $(U_\xi)_{\xi \in \partial T}$ of closed subgroups indexed by $\bd T$, called \textbf{root groups}, satisfying the following condition.
\smallskip
\begin{description}
\item[(Moufang condition)] $U_\xi$ fixes $\xi$ and acts sharply transitively%
\footnote{An action of a group on a set is called {\em sharply transitive} (also called {\em simply transitive} or {\em regular})
if it is transitive and free, i.e.\@ for every pair of elements $y,z$ in the set, there is a unique group element mapping $y$ to $z$.}
on $\bd T \setminus \{\xi\}$. 
\end{description}
\smallskip
In other words, this means that if $G$ is boundary-Moufang, then  $G$ and all the root groups are closed, and moreover the tuple  
\[ \MM = (\partial T, (U_\xi)_{\xi \in \partial T}) \]
is a Moufang set in the sense of Tits~\cite{Tits92}. 
 
 
It is clear that if $G$ is boundary-Moufang, then it is boundary-transitive. The following key observation, which follows relatively easily from general results on contraction groups due to Baumgartner--Willis~\cite{BaumgartnerWillis}, shows that the converse holds as soon as some contraction group is closed.
 
\begin{thmintro}\label{thm:ClosedContraction}
Let $G \leq \Aut(T)$ be boundary-transitive.  Assume that for some hyperbolic element $\alpha \in G$, the contraction group $U_\alpha$ is closed.

Then $G$ is boundary-Moufang. Moreover, the system of root groups in $G$ is unique in the following sense: for any conjugacy class $(V_\xi)_{\xi \in \partial T}$ of closed subgroups of $G$ satisfying the Moufang condition, the groups $V_\xi$ are conjugate to $U_\alpha$. 
\end{thmintro}

We shall moreover see that if $G$ is boundary-Moufang, then its root groups are subgroups of contraction groups. It is however not clear \emph{a priori} that if $G$ is boundary-Moufang, then the root groups coincide with the contraction groups; in particular, it is unclear whether the converse of Theorem~\ref{thm:ClosedContraction} holds.

Once Theorem~\ref{thm:ClosedContraction} is established, a major part of this paper consists in studying boundary-Moufang automorphism groups of $T$, relying on some results from the general theory of Moufang sets. The conclusions of Theorem~\ref{thm:padic} will be established for any boundary-Moufang group having torsion-free root groups (see Theorem~\ref{thm:TorsionFree} below). This provides a complete classification of Moufang sets at infinity of locally finite trees in characteristic~$0$.  

\subsection{Linearity criteria for general locally compact groups}

It turns out that the boundary-transitive groups of tree automorphisms appear quite  naturally in the study of the structure of general  locally compact groups, as illustrated by the following result, see \cite[Theorem~8.1]{CCMT}. 

\begin{thm*}[\cite{CCMT}]\label{thm:CCMT}
Let $G$ be a unimodular locally compact group without non-trivial compact normal subgroup. Assume that for some $\alpha \in G$, the closed subgroup $\overline{\la \alpha U_\alpha\ra}$ is cocompact in $G$. 

Then either $G$ is an almost connected rank one simple Lie group, or $G$ has a continuous, proper, faithful action on a locally finite tree $T$ which is $2$-transitive on $\bd T$. 
\end{thm*}

Combining the theorem above with the results proved in this paper, we deduce the following, where by definition, the \textbf{parabolic subgroup} associated to $\alpha \in G$ is  defined as 
\[
P_\alpha = \{ g \in G \mid (\alpha^n g \alpha^{-n})_{n \geq 0} \text{ is bounded}\}.
\]

\begin{corintro}\label{cor:Hil5}
Let $G$ be a unimodular locally compact group without non-trivial compact normal subgroup. Assume that for some $\alpha \in G$, the closed subgroup $\overline{\la \alpha U_\alpha\ra}$ is cocompact in $G$. 

Then the following holds.

\begin{enumerate}[\rm (i)]
\item If the parabolic group $P_\alpha$ is metabelian, then there is a non-discrete locally compact field $k$ and a continuous embedding $\PSL_2(k) \leq G \leq \PGL_2(k)$. 

\item 
If the contraction group $U_\alpha$ is closed and torsion-free, then there is a non-discrete locally compact field $k$ of characteristic $0$,  a semisimple algebraic $k$-group $\mathbf G$ of $k$-rank one and a continuous open homomorphism $\mathbf G(k) \to G$ whose image $G^+$ is a characteristic subgroup of finite index.   
\end{enumerate}
\end{corintro}

Notice that the converse statements of Corollary~\ref{cor:Hil5} also hold, namely the projective groups over $k$, or the rank one simple groups over $p$-adic fields, satisfy the algebraic conditions appearing in the statement. Thus Corollary~\ref{cor:Hil5} provides an algebraic/topological characterization of those groups within the category of locally compact groups, in the spirit of Hilbert's fifth problem.

\subsection*{Acknowledgement} We thank both referees, whose comments and suggestions were greatly appreciated, and improved the exposition of the results.

\section{From closed contraction groups to Moufang sets}


\subsection{Transitivity implies double-transitivity}
 
The following basic rigidity result can be found as Lemma~3.1.1 in \cite{BurgerMozes:trees1}; it can be generalised to isometry groups of Gromov hyperbolic metric spaces, see \cite[Theorem~8.1]{CCMT}.

\begin{prop}\label{prop:2trans}
Let $G \leq \Aut(T)$ be closed non-compact subgroup. If the $G$-action is transitive on $\bd T$, then it is doubly transitive.  In particular $G$  is compactly generated.  \qed
\end{prop}

Recall our convention that $T$ has no vertex of valency one. Observe that if $G$ is doubly transitive on $\bd T$, then the $G$-action on $T$ is cocompact, and any geodesic path joining two vertices of valency~$>2$ at minimal distance is a fundamental domain. In other words, upon discarding vertices of valency~$2$ in $T$, the $G$-action is edge-transitive.

\subsection{Structure of boundary-transitive groups}

We define the  \textbf{monolith} of a topological group to be the (possibly trivial) intersection of all its non-trivial closed normal subgroups. A subgroup $G \leq \Aut(T)$ acting transitively on $\bd T$ is called \textbf{boundary-transitive}. 
The following result can be extracted from the work of M.~Burger and Sh.~Mozes~\cite{BurgerMozes:trees1}.

\begin{thm}\label{thm:BurgerMozes}
Let $G \leq \Aut(T)$ be closed non-compact subgroup which is boundary-transitive.

Then the 
monolith $G^{(\infty)}$ of $G$ is compactly generated and  topologically simple. Moreover it acts 
doubly-transitively on $\bd T$. 
\end{thm}

\begin{proof}
The fact that $G^{(\infty)}$ is topologically simple is in \cite[Proposition 3.1.2]{BurgerMozes:trees1}; note that the condition of being locally $\infty$-transitive follows from \cite[Lemma~3.1.1]{BurgerMozes:trees1}.
By \cite[Proposition 1.2.1]{BurgerMozes:trees1}, $G / G^{(\infty)}$ is compact, and hence it is compactly generated.
The fact that $G^{(\infty)}$ acts 
doubly transitively on $\bd T$ also follows from \cite[Lemma~3.1.1 and Proposition~3.1.2]{BurgerMozes:trees1}.
\end{proof}

\subsection{Relation to parabolic subgroups}

Recall from the introduction that to each element $\alpha$ of a totally disconnected locally compact group $G$, one associates a \textbf{parabolic subgroup} defined by 
\[
P_\alpha = \{ g \in G \mid (\alpha^n g \alpha^{-n})_{n \geq 0} \text{ is bounded}\}.
\]
Clearly we have $U_\alpha \leq P_\alpha$. Moreover, it follows from \cite[Proposition~3]{Willis94} that $P_\alpha$ is closed in $G$. In our setting, we record the following geometric interpretation of parabolic subgroups. 

\begin{lem}\label{lem:parab}
Let $G \leq \Aut(T)$ be closed and boundary-transitive, and $\alpha, \beta \in G$ be  hyperbolic elements having the same  repelling fixed point  $\xi \in \bd T$. Then we have $U_\alpha = U_\beta$ and  $P_\alpha = P_\beta = G_\xi$. 
\end{lem}

\begin{proof}
 Let $\ell \subset T$ be the $\alpha$-axis and $\ell_- \subset \ell$ be a ray pointing to $\xi$. Observe that for each $x \in \ell$ we have $\alpha^{-n}x \in \ell_-$ for all sufficiently large $n\geq 0$. Therefore, given $g \in G$, the sequence $(\alpha^n g \alpha^{-n})_{n \geq 0}$ is bounded in $G$ if and only if the displacement function $x \mapsto d(x, g.x) $ is bounded on $\ell_-$. This in turn is equivalent to the assertion that $g$ maps $\ell_-$ at bounded Hausdorff distance from itself or, equivalently, that $g$ fixes $\xi$, thereby showing that $P_\alpha = G_\xi$.

The equality  $P_\alpha = P_\beta$ is now immediate. To check that  $U_\alpha = U_\beta$, observe that an element $u \in G$ belongs to $U_\alpha$ if and only if $u$ fixes pointwise a ray $\rho$ pointing to $\xi$ and, for each $n$, a ball $B(\rho(n), r_n)$ such that $\lim_n r_n =\infty$. Since the latter geometric description is independent of $\alpha$, the equality $U_\alpha = U_\beta$ follows.
\end{proof}

\subsection{The closure of a contraction group}

The following result is due to U.~Baumgartner and G.~Willis \cite{BaumgartnerWillis}. 

\begin{thm}\label{thm:BaumWillis}
Let $G$ be a second countable totally disconnected locally compact group and let $\alpha \in G$. Let $U_\alpha$ be the associated contraction group, let $P_{\alpha\inv}$ be the parabolic subgroup associated to $\alpha\inv$. Then we have the following. 

\begin{enumerate}[\rm (i)]
\item The closure of $U_\alpha$ in $G$ coincides with 
$
\overline{U_\alpha} = (\overline{U_\alpha} \cap P_{\alpha\inv} ) \cdot U_\alpha.
$ 
\item $U_\alpha$ is closed in $G$ if and only if $\overline{U_\alpha} \cap P_{\alpha\inv}= 1$. 

\item $P_\alpha = (P_\alpha \cap P_{\alpha\inv}) \cdot U_\alpha$. 
\end{enumerate}
\end{thm}

\begin{proof}
(i) follows from \cite[Lemma~3.29 and Corollary~3.30]{BaumgartnerWillis}. The assertion (ii) is proved in Theorem~3.32  while (iii) follows from Corollary~3.17 from \emph{loc.~cit.}.
\end{proof}

\subsection{Closed contraction groups and root groups}

\begin{prop}\label{prop:TransitiveContraction}
Let $G \leq \Aut(T)$ be closed and boundary-transitive, and $\alpha \in G$ be a hyperbolic element with repelling fixed point  $\xi \in \bd T$. Then:
\begin{enumerate}[\rm (i)]
\item  The contraction group $U_\alpha$ is transitive on $\bd T \setminus \{\xi\}$. 

\item If $U_\alpha$ is closed, then $U_\alpha$ is sharply transitive on $\bd T \setminus \{\xi\}$. 
\end{enumerate}

\end{prop}

\begin{proof}
(i) Since $\alpha$ is hyperbolic, it has exactly two fixed points in $\bd T$, namely an {attracting} fixed point $\xi_+$ and a {repelling} fixed point $\xi = \xi_-$. By Lemma~\ref{lem:parab}, we have $P_{\alpha} = G_{\xi_-}$ and $P_{\alpha\inv} = G_{\xi_+}$. From Proposition~\ref{thm:BaumWillis}(iii), we infer that $G_{\xi_-} = H.U_\alpha$, where $H = G_{\xi_+, \xi_-}$ is the pointwise stabiliser of the pair $\xi_+, \xi_-$. Notice that $G$ is not compact since it contains hyperbolic elements. Therefore, by Proposition~\ref{prop:2trans} the group $G_{\xi_-}$ is transitive on $\bd T \setminus \{\xi_-\}$.  Since $H$ fixes $\xi_+$, we infer that $U_\alpha$ is transitive on $\bd T \setminus \{ \xi_- \}$, as desired. 

\medskip \noindent (ii)
Since  $U_\alpha$ is closed, it follows from Proposition~\ref{thm:BaumWillis}(ii) that $U_\alpha \cap G_{\xi_+} = U_\alpha \cap P_{\alpha\inv} = 1$ and, hence, that the $U_\alpha$-action on $\bd T \setminus\{\xi_-\}$ is  sharply transitive, as desired. 
\end{proof}

\begin{cor}\label{cor:AbelianContraction}
Let $G \leq \Aut(T)$ be closed and boundary-transitive, and $\alpha \in G$ be a hyperbolic element. If the contraction group $U_\alpha$ is abelian, then it is closed. 
\end{cor}

\begin{proof}
Let $\xi_+, \xi_-$ be the attracting and repelling fixed points of $\alpha$, respectively. If $U_\alpha$ is abelian, so is its closure $\overline{U_\alpha}$. In particular the intersection $\overline{U_\alpha} \cap P_{\alpha \inv}$ is normal in $\overline{U_\alpha}$. By Lemma~\ref{lem:parab} we have $P_{\alpha\inv} = G_{\xi_+}$. Therefore, Proposition~\ref{prop:TransitiveContraction} ensures that $\overline{U_\alpha} \cap P_{\alpha \inv} = \overline{U_\alpha} \cap G_{\xi_+}$ fixes $\bd T$ pointwise, and is thus trivial. Theorem~\ref{thm:BaumWillis}(ii) then implies that   $U_\alpha$ must indeed be closed in $G$. 
\end{proof}



The final ingredient in the proof of Theorem~\ref{thm:ClosedContraction} is the following. 

\begin{lem}\label{lem:RootGroups:contraction}
Let $G \leq \Aut(T)$ be boundary-Moufang with associated root groups $(U_\eta)_{\eta \in \bd T}$. 
Let $\alpha \in G$ be a hyperbolic element with repelling fixed point $\xi \in \bd T$. Then:
\begin{enumerate}[\rm (i)]
\item The root group $U_\xi$ is contained in the contraction group $U_\alpha$. 

\item If $U_\alpha$ is closed, then $U_\xi = U_\alpha$. 
\end{enumerate}
\end{lem}

\begin{proof}
(i)
We have to show that for any $u \in U_\xi$, the only accumulation point of the sequence $(\alpha^n u \alpha^{-n})_{n \geq 0}$ is the identity. So let $v \in G$ be such an accumulation point.
Notice that $U_\xi$ is closed, and it is normal in $G_\xi$ since the root groups $U_\xi$ make up a full conjugacy class.
It thus follows that $v$ belongs to $U_\xi$.

Given any vertex $z$ on the $\alpha$-axis $\ell$, the vertex $v.z$ coincides with $\alpha^{n} u \alpha^{-n}.z$ for infinitely many $n$'s. Since $\alpha$ translates $\ell$ towards the boundary point $\xi_+$ of $\ell$ opposite to 
$\xi$ and since $u$ fixes pointwise a geodesic ray contained in $\ell$ and pointing to $\xi$, it follows that  $\alpha^{n} u \alpha^{-n}.z = z$ for all large $n$. Thus $v$ fixes $z$. We infer that $v$ fixes $\xi_+$, whence $v = 1$ in accordance with the Moufang condition. This confirms that the root group $U_\xi$ is contained in $U_\alpha$. 

\medskip \noindent (ii)
If $U_\alpha$ is closed, then it acts sharply transitively on $\bd T \setminus \{\xi\}$ by Proposition~\ref{prop:TransitiveContraction}(ii). Since $U_\xi $ is contained in $U_\alpha$ and acts transitively on   $\bd T \setminus \{\xi\}$, the equality $U_\xi = U_\alpha$ follows.
\end{proof}


\begin{proof}[Proof of Theorem~\ref{thm:ClosedContraction}]
By Proposition~\ref{prop:TransitiveContraction}(ii), the group  $U_\alpha$ is sharply transitive on $\bd T \setminus\{\xi_-\}$. Observe moreover that $U_\alpha$ is normal in $P_\alpha$. Therefore, for all $g \in G$, the group $g U_\alpha g\inv$ depends only on the boundary point $g.\xi_-$ (compare Lemma~\ref{lem:parab}). It follows that $G$ is boundary-Moufang with root groups coinciding with the conjugates of $U_\alpha$. 

The uniqueness of the system of root groups follows from Lemma~\ref{lem:RootGroups:contraction}.
\end{proof}

\subsection{Closed contraction groups and half-trees}

We end this subsection with the following easy observation, which implies that in a boundary-transitive group $G \leq \Aut(T)$ which satisfies Tits-independence property (called Property~(P) in~\cite{Tits:trees}), the contraction group of any hyperbolic element is never closed.

\begin{lem}\label{lem:half-tree}
Let $G \leq \Aut(T)$ be closed and boundary-transitive, and $\alpha \in G$ be a hyperbolic element such that the contraction group $U_\alpha \leq G$ is closed. 

Then for each half-tree $h \subset T$, the pointwise stabiliser of $h$ in $G$ is trivial.
\end{lem}

\begin{proof}
Let $\ell$ be a geodesic line of $T$ which is entirely contained in $h$. Since $G$ acts $2$-transitively on $\bd T$ by Proposition~\ref{prop:2trans}, we may assume, upon replacing $\alpha$ by some $G$-conjugate, that $\ell$ is the $\alpha$-axis. Let $h^* \subset T$ be the half-tree complementary to $h$. Then the projection of $h^*$ to $\ell$ is a single vertex $v$, and for any vertex $w \in \ell$ at distance $n$ from $v$, the pointwise stabiliser $\Fix_G(h)$ of $h$ in $G$ fixes pointwise the $n$-ball around $w$. This implies that $\Fix_G(h)$ is contained in the contraction group $U_\alpha$.  But by Proposition~\ref{prop:TransitiveContraction}, the group $U_\alpha$ is sharply transitive on $\bd T \setminus \{ \xi \}$, where $\xi$ is the repelling fixed point of $\alpha$. In particular $U_\alpha \cap \Fix_G(h) = 1$. Thus $\Fix_G(h)$ is trivial. 
\end{proof}
 
\section{The boundary-Moufang condition}

We now start our general study of closed subgroups $G \leq \Aut(T)$ which are boundary-Moufang. We assume throughout that a conjugacy class  $(U_\xi)_{\xi \in \bd T}$ of closed root groups is given, satisfying the Moufang condition from the introduction. 
Following the standard terminology in the theory of Moufang sets, the (possibly non-closed) normal subgroup 
\[
G^+ = \la U_\xi \mid \xi \in \bd T\ra \leq G
\]
is called the \textbf{little projective group}. We shall follow the standard notation and terminology on Moufang sets; this as well as much additional useful information can be consulted in \cite{DeMedtsSegev:course}, which provides an introduction to the  general theory of Moufang sets (independently of tree automorphism groups).


We however begin this section by recalling some of the most basic notions from the theory of Moufang sets that we will need later.

 \subsection{Abstract Moufang sets}

A Moufang set is a set $X$ together with a collection of groups $U_x \leq \Sym(X)$, one for each $x \in X$, such that
\begin{itemize}
    \item
	each $U_x$ fixes $x$ and acts sharply transitively on $X \setminus \{ x \}$;
    \item
	for every $x \in X$ and every $g \in G^+ := \langle U_z \mid z \in X \rangle$, we have $U_x^g = U_{g.x}$,
	where $U_x^g$ denotes conjugation in $\Sym(X)$, and where $g.x$ is the image of $x$ under the action of $g$.
\end{itemize}
The groups $U_x$ are called the {\bf root groups} of the Moufang set, and the group $G^+$ is called the {\bf little projective group}.
We will use the notation $U_x^*$ to denote the set of non-trivial elements of $U_x$.

Observe that $G^+$ is already generated by any two of the root groups. It is customary to assume that two of the elements of $X$
are called $0$ and $\infty$, so that $G^+ = \langle U_\infty, U_0 \rangle$.
An important feature of Moufang sets is given by the so-called $\mu$-maps.

\begin{prop}\label{prop:mu}
	For every $a \in U_0^*$, there are unique elements $b,c \in U_\infty^*$ such that $bac$ interchanges $0$ and $\infty$.
	We denote the product $bac$ by $\mu_a$, and we call it the $\mu$-map associated to $a$.
\end{prop}
\begin{proof}
	See, for instance, \cite[Proposition 4.1.1]{DeMedtsSegev:course}.
\end{proof}

Another observation is the fact that a Moufang set is completely determined by the group structure of one of its root groups $U_\infty$
together with the action of a {\em single} $\mu$-map, denoted by $\tau$;
we denote the corresponding Moufang set by
\[ \MM(U_\infty, \tau) . \]
The point here is that in order to describe the Moufang set, we only need $U_\infty$ as an abstract group (not as a permutation group)
and $\tau$ as a permutation of the elements of~$U_\infty^*$; the set $X$ can be recovered as $U_\infty \cup \{ \infty \}$,
and the other root groups can be recovered by conjugating with the appropriate elements; in particular $U_0$ can
be recovered as $U_\infty^\tau$.

For each $a \in U^*$, we define the {\bf Hua map} $h_a$ as the composition%
\footnote{We are using left actions, and therefore the composition of actions is from right to left.
Some care is needed, because most of the literature on Moufang sets uses right actions, and then composition of actions is from left to right.}
$\mu_a \tau$.
Observe that each $h_a$ fixes the elements $0$ and $\infty$.
We define the {\bf Hua subgroup} $H$ as the subgroup of $G^+$ generated by the Hua maps;
one of the main results of \cite{DeMedtsWeiss06} states that $H$ coincides with the two point stabilizer $G^+_{0,\infty}$;
see also \cite[Lemma 4.2.2]{DeMedtsSegev:course}.
A Moufang set is called {\bf proper} if $H \neq 1$, or equivalently, if $G^+$ is not sharply doubly transitive.

\begin{example}
	Let $k$ be a field or skew field, and let $U = (k, +)$ be the additive group of~$k$.
	Define a permutation $\tau \colon U^* \to U^* \colon x \mapsto -x^{-1}$.
	Then $\MM(U, \tau)$ is a Moufang set, and its little projective group is isomorphic to $\PSL_2(k)$.
\end{example}

We shall need the following example in the sequel.
\begin{example}\label{ex:skewherm}
	Let $k$ be a field or skew field with $\chr(k) \neq 2$, and let $\sigma$ be an involution of $k$
	(which may or may not be trivial, and which may or may not fix the center of $k$).
	Let $X$ be a (right) vector space over $k$, and suppose that $h$ is a non-degenerate anisotropic skew-hermitian form on $X$,
	i.e.\@ $h$ is bi-additive, and for all $x,y \in X$ and $s,t \in k$, we have $h(xs, yt) = s^\sigma h(x,y) t$
	and $h(y,x) = -h(x,y)^\sigma$.
	Moreover, let $\pi \colon X \to k \colon x \mapsto h(x,x) / 2$;
	since $h$ is anisotropic, $\pi(x)$ is never contained in $\Fix_k(\sigma)$ unless $x = 0$.
	Define a group $U$ with underlying set
	\[ U = \{ (x,s) \in X \times k \mid \pi(x) - s \in \Fix_k(\sigma) \} \,, \]
	and with composition given by
	\[ (x, s) \cdot (y, t) = (x + y, s + t + h(y,x)) \,. \]
	Finally, define a permutation
	\[ \tau \colon U^* \to U^* \colon (x, s) \mapsto (x s^{-1}, -s^{-1}) \,. \]
	Then $\MM(U, \tau)$ is a Moufang set, and its little projective group is isomorphic to the projective special unitary group corresponding to $h$.

	(We have deduced our formulas from \cite{TitsWeiss}; in particular, the description of $U$ is taken from
	\cite[Chapter 11]{TitsWeiss}, and the formula for $\tau$ can be computed
	from the formula for $\mu$ in \cite[(32.9)]{TitsWeiss}.)
\end{example}

 \subsection{Root groups are locally elliptic}

We now come back to the geometric set-up where $G \leq \Aut(T)$ is closed and boundary-Moufang. Remark that $G$ can be larger than the little projective group $G^+$. In fact, by definition $G$ is closed in $\Aut(T)$ and the root groups are closed in $G$, but it is not clear \emph{a priori} that the little projective group $G^+$ is closed. 

We start by collecting a number of useful basic observations. 

\begin{lem}\label{lem:RootGroupElliptic}
Every root group element fixes pointwise some geodesic ray. In particular every root group is a countable union of a directed chain of compact open subgroups. Moreover $G^+$ preserves the canonical bi-partition of the set of vertices of valency~$>2$; in particular it acts without inversion. 
\end{lem}
\begin{proof}
Observe that a hyperbolic automorphism of $T$ has exactly two fixed points in $\bd T$, and therefore it cannot belong to any root group. 

Fix a point $\infty \in \partial T$ and let $\rho$ be a ray pointing to $\infty$. Then every element of $U_\infty$ fixes pointwise a subray of $\rho$. If follows that $U_{\infty, \rho(n)} \leq U_{\infty, \rho(n+1)}$ for all $n\geq 0$, and moreover $U_\infty = \bigcup_{n\geq 0} U_{\infty, \rho(n)}$.

The last assertion follows since $G^+$ is generated by the root groups, and since every automorphism of $T$ fixing a vertex acts without inversion. 
\end{proof}

 \subsection{Closure of the little projective group }

\begin{prop}\label{prop:LittleProjMonolith}
Let $G \leq \Aut(T)$ be boundary-Moufang. 
Then the closure $\overline{G^+}$ of the little projective group coincides with the monolith $G^{(\infty)}$ of~$G$ (and is itself boundary-Moufang). In particular it is topologically simple, compactly generated and acts cocompactly on $T$. 
%
\end{prop}

\begin{proof}
Let $G^{(\infty)}$ be the monolith of $G$.
Since $\overline{G^+}$  is a closed normal subgroup of $G$, we have $G^{(\infty)} \leq \overline{G^+}$. For the reverse inclusion, observe that the monolith $G^{(\infty)}$ acts edge-transitively on $T$, and must therefore contain some hyperbolic element $\alpha$. 

Let $\pi \colon G \to G/G^{(\infty)}$ be the canonical projection. Since $\pi(\alpha)=1$, it follows that $\pi(U_\alpha)=1$.
(Indeed, if $g \in U_\alpha$, then $\lim_{n \to \infty} \alpha^n g \alpha^{-n} = 1$, and applying $\pi$ on this equality yields $\pi(g) = 1$.)
Thus the contraction group $U_\alpha$ is contained in $G^{(\infty)}$. By Lemma~\ref{lem:RootGroups:contraction}, this implies that all root groups are contained in $G^{(\infty)}$. Therefore, so is the little projective group $G^+$, and we are done. 
\end{proof}

\subsection{Continuity of the mu-maps}

Let us fix a pair of distinct points $0, \infty \in \bd T$ and denote by $\ell$ the geodesic line in $T$ joining $0$ to $\infty$. Let $\mu \colon U_0 ^* \to U_\infty^* U_0^* U_\infty^* \colon u \mapsto \mu_u$ be as in Proposition~\ref{prop:mu}. Thus $\mu_u$ swaps $0$ and $\infty$.
 
\begin{lem}\label{lem:mu-map}
Let $u \in U_0^*$ and $v \in \ell$ be the unique vertex $v \in \ell$  such that $u.\ell \cap \ell = [v, 0)$. Then  $\mu_u$ fixes $v$. 
\end{lem}
 
\begin{proof}
Set $\mu_u = u' u u''$ with $u', u'' \in U_\infty^*$. Lemma~\ref{lem:RootGroupElliptic} implies that each of the elements $u, u'$ and $u''$ fixes some point of $\ell$.
Notice that the fact that $u$ fixes some point of $\ell$ implies that it necessarily fixes $v$.

Now let $\ell'$ be the geodesic line joining $\infty$ to $u.\infty$ and $\ell''$ the geodesic line joining $\infty$ to $u\inv.\infty$. The definition of $v$ implies that $\ell \cap \ell' = \ell \cap \ell'' = [v, \infty)$. Since we have $u'u(\infty) = 0$ and $u''(0) = u\inv.\infty$, we deduce that $u'(\ell') = \ell$ and $u''(\ell) = \ell''$.
Since both $u'$ and $u''$ fix some point of $\ell$, this implies that both $u'$ and $u''$ fix $v$.
\end{proof}

\begin{lem}\label{lem:Mu}
For each vertex $v \in \ell$, there is an element $u \in U_0$ such that $\tau = \mu_u \in G^+$  fixes $v$ and swaps $0$ and $\infty$. 
\end{lem}

\begin{proof}
Let $\xi \in \partial T$ be a point such that the ray emanating from $v$ and pointing to $\xi$ intersects $\ell$ in $\{v\}$ only. Let $u \in U_\infty$ be the unique element mapping $0$ to $\xi$. Then $\mu_u \in G^+$ swaps $0$ and $\infty$ and fixes $v$ by Lemma~\ref{lem:mu-map}. Thus $\tau = \mu_u$ satisfies the desired conditions. 
\end{proof}

\begin{prop}\label{prop:mu-map:continuous}
The map $ \mu \colon U_0^* \to G : u \mapsto \mu_u$ is continuous. 
\end{prop}

\begin{proof}
Let $(u_n)_{n \geq 0}$ be a sequence of elements of $U_0^*$ converging to some $u \in U_0^*$. Let also $(u'_n)$ and $(u''_n)$ be sequences in $U_\infty^*$ such that $\mu(u_n) = u'_n u_n u''_n$.  Let also $u', u'' \in U_\infty^*$ be such that $\mu_u = u' u u''$. We need to show that $(u'_n)$ converges to $u'$ and $(u''_n)$ converges to $u''$. 

Let $v \in \ell$ be the unique vertex such that $u.\ell \cap \ell = [v, 0)$. Then $u$ fixes $v$ by Lemma~\ref{lem:mu-map}. Since $U_{0, v}$ is open in $U_0$, we deduce that $u_n$ belongs to $U_{0, v}$ for all sufficiently large $n$. Upon extracting we may assume that $u_n$ fixes $v$ for all $n$. By Lemma~\ref{lem:mu-map}, this implies that $u'_n$ and $u''_n$ both fix $v$ for all $n$. 
Since $U_{\infty, v}$ is compact, it suffices to show that any accumulation point of the sequence $(u'_n)$ (resp.\@ $(u''_n)$) coincides with $u'$ (resp.\@ $u''$). 

Let $(u''_{\psi(n)})$ be a subsequence converging to some $x''  \in U_{\infty, v}$. Upon extracting, we may assume that $(u'_{\psi(n)})$ also converges to some $x' \in U_{\infty, v}$. Now the sequence $\mu(u_{\psi(n)}) = u'_{\psi(n)} u_{\psi(n)} u''_{\psi(n)}$ converges to $x' u x''$. Since $\mu(u_{\psi(n)})$ swaps $0$ and $\infty$ for all $n$, so does the limit element $x' u x''$ since the $G$-action on $\bd T$ is continuous. By the uniqueness of $u'$ and $u''$, we infer that $x'=u'$ and $x'' = u''$. This implies that $(u''_n)$ converges to $u''$ and a similar argument shows that $(u'_n)$ converges to $u'$. 
\end{proof}

\subsection{The Hua subgroup}

The structure of the Hua group plays a key role in the general structure of a boundary-Moufang group. At this point, we record a basic result which holds more generally for groups that are doubly transitive on the boundary.

\begin{lem}\label{lem:Hua}
Let $G \leq \Aut(T)$ be closed and $2$-transitive on $\bd T$. 
Let $0,\infty \in \partial T$ be a pair of distinct points, and let $\ell \subset T$ be the geodesic line whose endpoints are $0$ and $\infty$. 

Then the group  $H = G_{0, \infty}$ decomposes as $H = H_c \rtimes \la t\ra$, where $H_c$ acts trivially on $\ell$ and $t$ acts as a translation. 

Moreover, if $G \leq \Aut(T)$ is boundary-Moufang, then the Hua subgroup $H = G^+_{0,\infty}$ decomposes as $H = H_c \rtimes \la t \ra$,
where $H_c$ acts trivially on $\ell$ and $t$ acts as a translation. 
\end{lem}

\begin{proof}
The only thing to prove is that $H$ contains a non-trivial translation along $\ell$. The result will then follow by choosing such a non-trivial translation of minimal possible translation length. 

Since $G$ is $2$-transitive on $\bd T$, it is non-compact and acts cocompactly on $T$. Therefore, it must contain a hyperbolic element $h$ since otherwise $G$ would fix a point in $T$ or an end in $\bd T$. An element of $G$ which maps the attracting (resp.\@ repelling) fixed point of $h$ to $0$ (resp.\@ $\infty$) must conjugate $h$ into $H$, thereby establishing the claim. 

In case $G$ is boundary-Moufang, we apply the same argument to $G^+$ or alternatively, we pick 
two adjacent vertices $v, v'$ on $\ell$ and  elements  $\tau$ and $\tau'$ in $G^+$ fixing respectively $v$ and $v'$ and swapping $0$ and $\infty$; see Lemma~\ref{lem:Mu}.  The product $t = \tau \tau'$ is then an element of $H = G_{0, \infty}  $  acting as a translation on $\ell$. 
\end{proof}

\begin{remark}
	Notice that when $G$ is boundary-Moufang, we do not know whether the little projective group $G^+$ is always closed in $\Aut(T)$.
	(However, see section~\ref{ss:G+simple} below.)
	This is the reason why we needed a separate statement for the boundary-Moufang case in the previous Lemma.
\end{remark}

It is an intriguing problem to show that the compact group $H_c$ is always infinite or, equivalently, that the Hua subgroup $H$ cannot be virtually cyclic. This question makes sense and is interesting for arbitrary Moufang sets. 

Right now, we only record the following immediate consequence of Lemma~\ref{lem:Hua}.

\begin{cor}\label{cor:sharply}
A group $G \leq \Aut(T)$ acting $2$-transitively on $\bd T$ is never sharply $2$-transitive.
Similarly, if  $G \leq \Aut(T)$ is boundary-Moufang, then $G^+$ is not sharply $2$-transitive on $\partial T$. 
\end{cor}
\begin{proof}
Indeed, the translation $t \in G_{0, \infty}$ (or $t \in G^+_{0,\infty}$, respectively) appearing in Lemma~\ref{lem:Hua} is a non-trivial element of $G$ (or $G^+$, respectively) which fixes two points in $\partial T$.
\end{proof}

\section{Structure of the root groups}

 \subsection{Fixed points of root subgroups}

Our next goal is to give a more precise description of the fixed-point set of a root group element. 
We already know from Lemma~\ref{lem:RootGroupElliptic} that every root group element fixes pointwise some geodesic ray,
but the following statement is more detailed.

\begin{lem}\label{lem:psi(n)}
Let $G \leq \Aut(T)$ be boundary-Moufang.
Then there is a non-decreasing map $\psi\colon \NN \to \NN$ with $\psi(0) = 0$ and  $\psi(n) \geq n$ such that for all $\infty \in \bd T$ and $v \in T$, the group $U_{\infty, v}$ fixes pointwise the $n$-ball around every vertex $x \in (\infty, v]$ with $d(v, x) \geq \psi(n)$. 
%
%
\end{lem}

We shall use the following basic fact.

\begin{lem}\label{lem:UnifContraction}
Let $U$ be a locally compact group and $\alpha \in \Aut(U)$ be such that 
\[ \lim_{n \to \infty} \alpha^n(u)=1 \]
for all $u \in U$. Then $\alpha $ is uniformly contracting in the following sense: for each compact subset $K \subset U$ and each identity neighbourhood $V \subset U$ there is some $N$ such that $\alpha^n(K) \subset V$ for all $n >N$. 
\end{lem}
\begin{proof}
See \cite[Lemma~1.4(iv)]{Wang84}. 
\end{proof}

\begin{proof}[Proof of Lemma~\ref{lem:psi(n)}]
%
Let $\infty \in \bd T$. By Lemma~\ref{lem:RootGroups:contraction}, the root group $U_\infty$ is contained in the contraction group $U_\alpha$ of every hyperbolic element $\alpha$ having $\infty$ as a repelling fixed point. Lemma~\ref{lem:Hua} implies that $G$ contains such a hyperbolic element. 

Now Lemma~\ref{lem:UnifContraction} implies that for each vertex $v$ of $T$, there is a non-decreasing map $\psi_{\infty, v} \colon \NN \to \NN$ with $\psi_{\infty, v}(n) \geq n$, such that the group $U_{\infty, v}$ fixes pointwise the $n$-ball around every vertex $x \in (\infty, v]$ with $d(v, x) \geq \psi_{\infty, v}(n)$.  

Since $U_\infty$ is transitive on $\bd T \setminus \{\infty\}$, we deduce that $ \la \alpha\ra U_\infty$ acts cocompactly on~$T$. Let $v_1, \dots, v_n$ be a finite set of representatives of $ \la \alpha\ra U_\infty$-orbits of vertices. Set $\psi_{\infty}(n) = \max \{\psi_{\infty, v_i}(n) \mid i=1, \dots, n\}$. This map satisfies the property that for all vertices $v$,  the group  $U_{\infty, v}$ fixes pointwise the $n$-ball around every vertex $x \in (\infty, v]$ with $d(v, x) \geq \psi_{\infty}(n)$.

The desired assertion now follows from the fact that $G$ is boundary-transitive,  and hence the map $\psi = \psi_\infty$ satisfies  the requested property.
\end{proof}

\subsection{Closed contraction groups}

Given a locally compact group $U$, we denote by $\Aut(U)$ the group of continuous automorphisms of $U$. An automorphism $\alpha \in \Aut(U)$ is called \textbf{contracting} if $\lim_{n \to \infty} \alpha^n(u)=1$ for all $u \in U$. 
The following result is due to H.~Gl\"ockner and G.~Willis. 

\begin{thm}\label{thm:GlocknerWillis}
Let $U$ be a totally disconnected locally compact group and $\alpha \in \Aut(U)$ be a contracting automorphism. 

Then the set $T$ of torsion elements of $U$ is a closed characteristic subgroup. Moreover there are some $r \geq 0$ and some primes $p_1, \dots, p_r$ such that $U$ has an $\alpha$-invariant decomposition $U \cong V_1 \times \dots \times V_r \times T$,
where each $V_i$ is a nilpotent $p_i$-adic analytic group.
\end{thm}

\begin{proof}
See \cite[Theorem~B]{GloecknerWillis}.
\end{proof}

We point out the following consequence of the work by H.~Gl\"ockner and G.~Willis.  

\begin{cor}\label{cor:contraction_abelianization}
Let $H = \la \alpha \ra \ltimes U$ be the semi-direct product of a totally disconnected locally compact group $U$ with the cyclic group generated by a contracting automorphism $\alpha\in \Aut(U)$. 

Then $U = [H, H]$. In particular the derived group $[H, H]$ is open and closed in $H$. 
\end{cor}

\begin{proof}
We need to prove that any (abstract) homomorphism $\varphi \colon H \to A$ to an abelian group factors through $H/U$. The structure theory of contraction groups developed by Gl\"ockner--Willis~\cite{GloecknerWillis} ensures that $U$ has a finite $\alpha$-stable composition series, all of whose subquotients  are simple contraction groups. In particular, by induction it suffices to show that the homomorphism $\varphi \colon H \to A$ factors through $H/V$, where $V \leq U$ is a minimal $\alpha$-stable closed normal subgroup. 

Now $V$ is a simple contraction group under the $\alpha$-action, and \cite[Theorem~A]{GloecknerWillis} shows that there are two possibilities: either $V$ is isomorphic to a finite-dimensional $p$\nobreakdash-adic vector space for some prime $p$, and $\alpha$ acts as a contracting linear transformation, or there is a finite simple group $F$ such that $V \cong (\bigoplus_{n <0} F) \times (\prod_{n \geq 0} F)$ and $\alpha$ acts as the positive shift. 

In the first case, it is a matter of elementary linear algebra to check that for all $v \in V$, the equation $[\alpha, x] = v$ has a solution $x \in V$: the point is that $\alpha$ is contracting, and hence does not have eigenvalues of norm~$1$. In the second case, the equation $[\alpha, x] = v$ also has a solution $x \in V$ for all $v \in V$. Indeed, decomposing that equation componentwise according to  $V \cong (\bigoplus_{n <0} F) \times (\prod_{n \geq 0} F)$, one finds that the equation is equivalent to the system of equations
\[
\begin{cases}
x_{j-1} x_j\inv = 1  & \text{ for all } j < i \\
x_{j-1} x_j\inv = v_j  & \text{ for all } j \geq i
\end{cases}
\]
where $i$ is the smallest index of a non-zero component of $v$. Clearly this system admits a (unique) solution. This confirms that $V \leq [H, H] \leq \ker \varphi$, thereby finishing the proof.
\end{proof}

\subsection{The little projective group is simple}\label{ss:G+simple}

We have seen in Proposition~\ref{prop:LittleProjMonolith} that if $G \leq \Aut(T)$ is boundary-Moufang, then the closure $\overline{G^+}$ of the little projective group is topologically simple (and remains boundary-Moufang). One expects $G^+$ to be closed in $G$ (hence topologically simple) in full generality. At this point, we can only establish the following, the proof of which follows closely a well-known argument due to Tits \cite{Tits:64}. 

\begin{prop}\label{prop:LittleProjSimple}
Let  $G \leq \Aut(T)$ be boundary-Moufang with soluble root groups. Then the little projective group $G^+$ is simple (as an abstract group).
\end{prop}

\begin{proof}
If $N$ is a non-trivial normal subgroup, then it is transitive on $\bd T$. Thus $G^+ = N.G^+_\xi$ for any $\xi \in \bd T$. Since the root group $U_\xi$ is normal in $G^+_\xi$, it follows that any two conjugates of $U_\xi$ in $G^+$ can be conjugated by an element of $N$. Since $G^+$ is generated by those conjugates, we infer that $G^+ = N.U_\xi$. Therefore $G^+/N \cong U_\xi / (U_\xi \cap N)$. But $U_\xi$ is soluble by hypothesis, while $G^+$ is perfect by Corollary~\ref{cor:contraction_abelianization}. Therefore $G^+/N$ must be trivial.
\end{proof}

\subsection{Abelian contraction groups of prime exponent}

It follows from Theorem~\ref{thm:GlocknerWillis} that if $U$ is an abelian contraction group which is torsion-free and locally pro-$p$ for some prime $p$, then $U$ is a finite-dimensional topological vector space over $\QQ_p$. The following observation is an analogue of that fact in positive characteristic. 

\begin{lem}\label{lem:PrimeExponent}
Let $p$ be a prime,  $U$ be an additive abelian totally disconnected locally compact group of exponent $p$ and $\alpha \in \Aut(U)$ be a contracting automorphism.

Then there is a finite-dimensional topological vector space $V$  over  $\FF_p(\!(t)\!)$ and a topological isomorphism $\varphi : U \to V$ such that $\varphi  \alpha \varphi\inv$ acts on $V$ as scalar multiplication by~$t$. 

In particular, the centraliser $\centra_{\End(U)}(\alpha)$ of $\alpha$ in the ring of continuous endomorphisms of $U$ is isomorphic to $\End_{ \FF_p(\!(t)\!)}(V)$. 
\end{lem}
 
\begin{proof}
Notice that by hypothesis $U$ is canonically endowed with the structure of a vector space over $\FF_p$. 

We claim that, for each $n \in \ZZ$, each sequence  $(a_i)_{i \geq n}$  in $\FF_p$ and all $u \in U$, the series  $\sum_{i=n}^\infty a_i \alpha^i(u)$ converges in $U$.

\medskip
In order to establish the claim, we let $V_0 \leq U$ be a compact open subgroup such that $\alpha(V_0) \leq V_0$. Such a compact open subgroup always exists, see e.g.\@ \cite[Proposition~1.1(b)]{GloecknerWillis}. For all $n \in \ZZ$, set $V_n = \alpha^n(V_0)$.  By Lemma~\ref{lem:UnifContraction}, the automorphism $\alpha$ uniformly contracts every compact subset of $U$ to the identity. Therefore we have $\bigcap_{n\geq 0} V_n = 1$. In particular $(V_n)_{n \geq 0}$ is a base of identity neighbourhoods. 

For $u, v \in U$, define $d(u,v) = 2^{-n}$ with $n = \sup\{k \mid v-u \in V_k\}$. Then $d$ is a $U$-invariant ultra-metric on $U$.   Since $(V_n)_{n \geq 0}$ is a base of identity neighbourhoods, the group topology on $U$ coincides with the topology defined by $d$. 

Now it is immediate to verify that the sequence $(\sum_{i=n}^m a_i \alpha^i(u))_{m \geq n}$ is Cauchy. Since a metrisable locally compact group is complete, the series $\sum_{i=n}^\infty a_i \alpha^i(u)$ converges as desired. The claim stands proven. 

\medskip
We next claim that  the map
\[
s \colon \FF_p(\!( t ) \!) \times U  \to U \colon \Bigl( \sum_{i=n}^\infty a_i t^i, u \Bigr) \mapsto \sum_{i=n}^\infty a_i \alpha^i(u)
\]
is continuous and turns $U$ into a topological vector space of dimension $d < \infty$ over the local field $\FF_p(\!( t ) \!)$. 

Indeed, for each $u \in U$ and $n \in \ZZ_{\geq 0}$, the inverse image under the scalar multiplication $s$ of $u + V_n$ contains $1+t^m \FF_p[\![ t ] \!] \times (u+V_n)$, where
$m = \max \{ 0, n + \log_2 d(0,u) \}$.
This proves that $s$ is continuous. That $s$ satisfies the axioms of a scalar multiplication is easy to verify.

We have $[V_0:V_1] = p^d$ for some $d$.
Let $e_1, \dots, e_d \in V_0$ be such that $\la e_1, \dots, e_d \ra V_1 = V_0$. Then $\la \alpha^n(e_i) \mid n \geq 0; \;  i=1, \dots, d\ra$ is dense in $V_0$ because it maps surjectively on each quotient $V_0/V_{n}$ for all $n \geq 0$. Moreover, for each $j$, the group $\la \alpha^n(e_i) \mid n \geq 0; \;  i \neq j\ra$ does not surject onto $V_0/V_1$. Together, these two observations imply that
\[
V_0 = \FF_p[\![ t ] \!] e_1 \oplus \dots \oplus \FF_p[\![ t ] \!] e_d.
\]
Since $U = \lim_{n \to \infty} \alpha^{-n} (V_0)$, this implies that
$\{e_1, \dots, e_d\}$ is a basis for $U$ over $ \FF_p(\!( t )\!)$, and $U$ is finite-dimensional as desired.
The claim is proven and the lemma follows. 
\end{proof}

\section{Structure of the vertex stabilizers}


By definition, the root groups of a group satisfying the boundary-Moufang condition are closed. Moreover, we have seen in Lemma~\ref{lem:RootGroups:contraction} that each root group is contracted by the conjugation action of some hyperbolic element. It follows that the structure of these can be analysed in the light of general results on contraction groups, some of which were recalled above. 

Our next goal is to relate the structure of the root groups to the local structure of a boundary-Moufang group $G \leq \Aut(T)$, i.e.\@ to the structure of \emph{open} subgroups of $G$.

In this section, we fix a pair  of distinct points $0, \infty \in \partial T$, we  let $\ell$ be the line of $T$ joining $0$ to $\infty$ and set $H= G_{0, \infty}$. Our goal is to prove the following. 

\begin{prop}\label{prop:pro-p}
If there is a prime $p$ such that for some vertex $v$ on $\ell$, the groups $U_{0, v}$ and $H_v$ are virtually pro-$p$ (resp.\@ $p$-adic analytic), then $G_v$ is virtually pro-$p$ ($p$-adic analytic) as well. 
\end{prop}

In order to establish this, we need to study how the vertex-stabiliser $G_v$ is generated by its intersection with the root groups.



\begin{prop}\label{prop:VertexStab}
For each vertex $v$ on $\ell$, we have a decomposition
\[ G_v =  U_{\infty, v} U_{0, v}  U_{\infty, v} U_{0,v} H_v = U_{0, v} U_{\infty, v} U_{0, v} U_{\infty,v} H_v . \]
\end{prop}

\begin{proof}
Let $\tau \in G$ be an element fixing $v$ and swapping $0$ and $\infty$ (see Lemma~\ref{lem:Mu}). Pick any $g \in G_v$ and consider the element $x = g.0$. If the vertex $v$ lies on the geodesic line $(x, \infty)$, then there is an element $u \in U_{\infty, v}$ such that 
$u.x = 0$, hence $ug.0 = 0$. If this is not the case, then  $v$ lies on the geodesic line $(x, 0)$ and, hence, $v$ lies also on the geodesic line $(\tau.x, \infty) = \tau.(x, 0)$. There is then an element $u' \in U_{\infty, v}$ such that $u' \tau . x = 0$, hence $u' \tau g.0 = 0$. 

In either case, we have shown that $g \in U_{\infty, v} G_{0, v} \cup \tau\inv U_{\infty, v} G_{0, v}$.
Since $\tau$ can be written as a product $u' u u''$ with $u \in U_{0, v}$ and $u', u'' \in U_{\infty, v}$ (see the proof of Lemma~\ref{lem:Mu}), we deduce that 
\[ G_v =  U_{\infty, v} U_{0, v}  U_{\infty, v} G_{0, v} . \]
Since $U_0$ is transitive on $\bd T \setminus \{0\}$, we have $G_0 = U_0 H$. Notice that $G_{0, v}$ consists of elliptic elements, and so does $U_0$ by Lemma~\ref{lem:RootGroupElliptic}. Thus all the elements of $G_{0, v} \cup U_0$ are contained in the kernel of the Busemann character centred at $0$. In view of Lemma~\ref{lem:Hua}, this implies that if $g = uh \in U_0 H$ fixes $v$, then $h$ acts trivially on $\ell$. In particular $h$ fixes $v$, and hence $u$ does so as well. This shows that 
\[ G_{0, v} = U_{0,v} H_v, \]
which, combined with the above, yields
\[ G_v =  U_{\infty, v} U_{0, v}  U_{\infty, v} U_{0,v} H_v . \]
A similar argument with the roles of $0$ and $\infty$ interchanged yields 
\[ G_v =  U_{0,v} U_{\infty, v} U_{0, v}  U_{\infty, v} H_v . \]
as requested.
\end{proof}

It is in fact useful to dispose of  similar decompositions for some open subgroups of $G_v$. The following result shows that some of these open subgroups are often subjected to a simpler decomposition than the one provided by Proposition~\ref{prop:VertexStab}. In order to state it,  it is convenient to define  the \textbf{$n^{\text{th}}$ congruence subgroup} of the stabiliser $G_v$ as the pointwise stabiliser $G_v^{(n)}$ of the $n$-ball around $v$. Moreover, for any subgroup $J \leq G_v$, we set $J^{(n)} = J \cap G_v^{(n)}$.

\begin{prop}\label{prop:VertexStab(n)}
Let $v, x_0, x_\infty$ be vertices of $\ell$ such that $x_0 \in (0, v]$, $x_\infty \in [v, \infty)$ and $x_0 \neq x_\infty$. 

For all $n\geq 0$, if $d(x_0, v) \geq \psi(n)$ and $d(x_\infty, v) \geq \psi(n)$ (where $\psi(n)$ is the map from Lemma~\textup{\ref{lem:psi(n)}}), then the product maps
\[
  U_{\infty, x_0}  \times U_{0, x_\infty}  \times H_v^{(n)} \to G_{x_0, x_\infty}^{(n)}
\]
 and 
\[
  U_{0,x_\infty} \times U_{\infty, x_0} \times H_v^{(n)} \to G_{x_0, x_\infty}^{(n)}
\]
 are bijective. 
\end{prop}

\begin{proof}
We let $\psi \colon \NN \to \NN$ be as in Lemma~\ref{lem:psi(n)}. The latter lemma implies that $U_{\infty, x_0} = U_{\infty, x_0}^{(n)}$ and  $ U_{0, x_\infty}  =  U_{0, x_\infty}^{(n)}$. 

Let now $g \in G_{x_0, x_\infty}^{(n)}$. Since $g$ fixes pointwise a segment $[x_0, x_\infty] \subset \ell$ of positive length, it follows that $g.0 \neq \infty $ and $g.\infty \neq 0$. 
We deduce that there is a unique $u \in U_{0}$ such that $ug.\infty = \infty$, and clearly $u$ fixes $x_\infty$, hence $u \in U_{0, x_\infty}$.
Similarly,  there is a unique $w \in U_{\infty, x_0}$ such that $wug.0 = 0$. Moreover we have $wug.\infty = w.\infty = \infty$. Thus $h = wug \in H_v$. Moreover $h$ fixes pointwise the $n$-ball around $v$ since $u$, $w$ and $g$ all do. 
Hence  $g = u\inv w\inv h \in U_{0, x_\infty}  U_{\infty, x_0} H_v^{(n)}$. Moreover this decomposition is unique in view of the uniqueness of $u$ and $w$ in the construction above. 

A similar argument works with $0$ and $\infty$ interchanged.
\end{proof}

\begin{proof}[Proof of Proposition~\ref{prop:pro-p}]
\medskip
Let $n>0$ be sufficiently large so that  $U_{0, v}^{(n)}$ and $H_v^{(n)}$ are pro-$p$ groups. By Lemma~\ref{lem:Mu}, there is an element  $\tau \in G^+$ which fixes $v$ and swaps $0$ and $\infty$. Then $(U_{0, v}^{(n)})^\tau = U_{\infty, v}^{(n)}$, which proves that the latter is also pro-$p$. 

Let $\psi \colon \NN \to \NN$ be the map provided by Lemma~\ref{lem:UnifContraction}. Let also  $x_0 \in (0, v]$ and $x_\infty \in [v, \infty)$ be such that $d(v, x_0) = d(v, x_\infty) = \psi(n)$. We shall show that $G_{x_0, x_\infty}^{(n)}$ is pro-$p$. Since $G_{x_0, x_\infty}^{(n)}$ is  an open subgroup of $G_v$, it will follow that $G_v$ is virtually pro-$p$. 

So let $m \geq n$. We have to show that the index of  $G_{x_0, x_\infty}^{(m)}$ in  $G_{x_0, x_\infty}^{(n)}$ is a power of $p$. Let $y_0 \in (0, v]$ and $y_\infty \in [v, \infty)$ be such that $d(v, y_0) = d(v, y_\infty) = \psi(m)$. Thus we have  $G_{y_0, y_\infty}^{(m)} \leq  G_{x_0, x_\infty}^{(m)}$  and it suffices to show that  $[ G_{x_0, x_\infty}^{(n)} : G_{y_0, y_\infty}^{(m)}]$ is a power of $p$. 

By Proposition~\ref{prop:VertexStab(n)}, 
the product maps
\[  U_{\infty, x_0}   \times U_{0, x_\infty}   \times H_v^{(n)} \to G_{x_0, x_\infty}^{(n)} \]
and 
\[
 U_{\infty, y_0}  \times U_{0, y_\infty} \times H_v^{(m)} \to G_{y_0, y_\infty}^{(m)}
\]
are both bijective. Therefore we have
\[
[ G_{x_0, x_\infty}^{(n)} : G_{y_0, y_\infty}^{(m)}] = 
[U_{\infty, x_0}^{(n)}: U_{\infty, y_0}^{(m)}  ]\; 
[U_{0, x_\infty}^{(n)}: U_{0, y_\infty}^{(m)}] \;
[H_v^{(n)} : H_v^{(m)} ],
\]
which is indeed a power of $p$.

If one now assumes additionally that $U_{0, v}^{(n)}$ and $H_v^{(n)}$ are in fact $p$-adic analytic, then so is $(U_{0, v}^{(n)})^\tau = U_{\infty, v}^{(n)}$. Moreover, it follows from \cite[Theorems~3.17 and~9.36]{DDMS} that a pro-$p$ group which is a finite (not necessarily direct) product of $p$-adic analytic groups  it is itself $p$-adic analytic. Since $G_{x_0, x_\infty}^{(n)}$ is pro-$p$ by the first part of the proof,  the desired conclusion follows again from the decomposition 
$G_{x_0, x_\infty}^{(n)} =   U_{\infty, x_0}^{(n)}  U_{0, x_\infty}^{(n)}  H_v^{(n)} $ which was established above.
\end{proof}

\section{Structure of the Hua subgroup}

The product decompositions established in Proposition~\ref{prop:VertexStab(n)} show that the local structure of a boundary-Moufang group is determined in an essential way by the structure of the root groups, and of the Hua subgroup. While meaningful information on the root groups can be deduced from the fact that they are contraction groups, the structure of the Hua subgroup seems to be much less tractable. We have seen in Lemma~\ref{lem:Hua} that the Hua subgroup $H$ has a maximal compact normal subgroup $H_c$ with $H/H_c$ infinite cyclic. It turns out to be an important question to identify interesting subgroups of the compact part $H_c$. The relevance of this question relies in the fact that the centraliser of such a subgroup happens to be the main source of sub-Moufang sets (see Section~\ref{sec:subMoufang} below), whose existence is useful in analysing the global structure of the given Moufang set. This will be illustrated below in our analysis of boundary-Moufang groups with torsion-free root groups. Of course, this whole approach can only work if $H_c$ indeed contains significant subgroups. In this direction, we shall now discuss the following conjecture, which predicts that $H_c$ must be infinite. 

\begin{conj}\label{conj:VirtCyclic:1}
Let $G \leq \Aut(T)$ be  closed, non-compact and boundary-transitive. Then the stabiliser $G_{\xi, \xi'}$ of a pair of distinct boundary points $\xi, \xi' \in \bd T$ is not discrete.
\end{conj}

We remark that a counter-example to Conjecture~\ref{conj:VirtCyclic:1} necessarily satisfies the Moufang condition. Indeed, we have the following. 

\begin{lem}\label{lem:EquivConj}
Let $G \leq \Aut(T)$ be  closed, non-compact and boundary-transitive. Assume that the stabiliser $G_{\xi, \xi'}$ of some pair of distinct boundary points $\xi, \xi' \in \bd T$ is discrete.

Then for each hyperbolic element $\alpha \in G$, the contraction group $U_\alpha$ is closed. In particular $G$ is boundary-Moufang. 
\end{lem}
\begin{proof}
Let $\alpha \in G$ be hyperbolic and $\xi_+, \xi_- \in \bd T$ be respectively its attracting and repelling fixed point. Such an element exists by Lemma~\ref{lem:Hua}, since $G$ is $2$-transitive on $\bd T$ by Proposition~\ref{prop:2trans}. 

By Lemma~\ref{lem:parab} we have $P_{\alpha\inv} = G_{\xi_+}$ and $\overline{U_\alpha} \leq G_{\xi_-}$.  By assumption the intersection $G_{\xi_+} \cap G_{\xi_-} = G_{\xi_+, \xi_-}$ is discrete. It is therefore finite-by-cyclic, hence cyclic-by-finite by Lemma~\ref{lem:Hua}.  Since every element of $\overline{U_\alpha}$ is elliptic, it follows that $\overline{U_\alpha} \cap P_{\alpha\inv}$ is a compact subgroup of $G_{\xi_+, \xi_-}$, and is therefore finite.
We infer that $U_\alpha \cap  P_{\alpha\inv}$ is a finite, hence closed, subgroup normalised by $\alpha$, on which $\alpha$ acts as a contracting automorphism. Thus $U_\alpha \cap  P_{\alpha\inv}$ is trivial, and hence so is $\overline{U_\alpha} \cap P_{\alpha\inv}$ by \cite[Lemma~3.31]{BaumgartnerWillis}.
By Theorem~\ref{thm:BaumWillis}, this implies that $U_\alpha $ is closed. Theorem~\ref{thm:ClosedContraction} then guarantees that $G$ is boundary-Moufang.
\end{proof}

In view of Lemma~\ref{lem:EquivConj},  Conjecture~\ref{conj:VirtCyclic:1} is equivalent to the following. 

\begin{conj}\label{conj:VirtCyclic:2}
Let $T$ be a locally finite tree and $G \leq \Aut(T)$ be boundary-Moufang. Then the stabiliser $G_{\xi, \xi'}$ of a pair of distinct boundary points $\xi, \xi' \in \bd T$ is not virtually cyclic.
\end{conj}


We shall now relate this to a conjecture in finite group theory, going back to O.~Kegel in the sixties.

\subsection{Products of nilpotent groups and the Kegel--Wielandt theorem}\label{sec:Kegel}

A well-known theorem by Kegel--Wielandt (see \cite{Kegel} and \cite{Wielandt}) asserts that if a finite group $G$ is a product of  two nilpotent subgroups $A, B$, then $G$ is soluble (and if $A \neq B$ then $G$ has a proper normal subgroup containing $A$ or $B$).  Whether the derived length of $G$ can be bounded in terms of the nilpotency classes of $A$ and $B$ is however an open problem going back to O.~Kegel~\cite{Kegel1965} (the conjecture mentioned by Kegel in \emph{loc.~cit.} is phrased slightly differently; we refer to \cite{CosseyStonehewer} for a more detailed discussion around this conjecture): 

\begin{conj}[O. Kegel]\label{conj:KW}
There is a function $f\colon \NN \times \NN \to \NN$ such that if a finite group $G$ is a product of two nilpotent groups  of respective nilpotency classes $a$ and $b$, then $G$ is soluble of derived length at most $ f(a, b)$. 
\end{conj}

The relation between Conjecture~\ref{conj:KW} and those from the preceding subsection is provided by Corollary~\ref{cor:Hua:NilpotentRootGps} below.
We will need the following lemma, which is a consequence of Proposition~\ref{prop:VertexStab(n)}.

\begin{lem}\label{lem:Hua:NilpotentRootGps}
Let $G \leq \Aut(T)$ be boundary-Moufang and let $0\neq \infty \in \bd T$.  

If the Hua subgroup $H=G_{0, \infty}$ is virtually cyclic, then some compact open subgroup $V \leq G$ has a product decomposition $V = (V \cap U_0)\cdot (V \cap U_\infty)$. In particular the little projective group $G^+$ is open of finite index in $G$.  
\end{lem}

\begin{proof}
By Lemma~\ref{lem:Hua} we have $H = H_c \rtimes \la t \ra$, where  $t$ is a  hyperbolic element and $H_c$ is compact and acts trivially on the geodesic line $\ell \subset T$  joining $0$ to $\infty$. 

Let $v \in \ell$ be any vertex. Thus we have $H_c = H_v$.
Since $H$ is virtually cyclic, $H_v$ is finite; hence
there exists some  $n \geq 0$ such that the $n^{\text{th}}$ congruence subgroup $H_v^{(n)}$ is trivial. By Proposition~\ref{prop:VertexStab(n)}, this implies that the stabiliser $G_v$ has some open subgroup $V$ of the form $V = G_{x_0, x_\infty}^{(n)}$ for some $x_0, x_\infty \in \ell$, which is the product of its intersections with the root groups $U_0$ and $U_\infty$. 

Since $U_0$ and $U_\infty$ are contained in the little projective group $G^+$, it follows that $V \leq G^+$, and hence $G^+$ is open, hence closed, in $G$. Since $G/G^+$ is compact by Proposition~\ref{prop:LittleProjMonolith}, it follows that $G^+$ has finite index.
\end{proof}

We deduce a sufficient condition for Conjecture~\ref{conj:VirtCyclic:2} to hold. 

\begin{cor}\label{cor:Hua:NilpotentRootGps}
Let $G \leq \Aut(T)$ be boundary-Moufang. 

If the root groups are nilpotent and if Conjecture~\ref{conj:KW} holds, then the Hua subgroup is not virtually cyclic. 
\end{cor}

\begin{proof}
Let $V \leq G$ be a compact open subgroup as in Lemma~\ref{lem:Hua:NilpotentRootGps}. Then $V$ is a product of two nilpotent groups,  and so is every finite quotient of $V$, with a uniform bound on the nilpotency classes of the factors. Conjecture~\ref{conj:KW} then implies that  $V$ is soluble. Thus $G$ has some open soluble subgroup. In particular, so does the intersection $G^{(\infty)}$ of all non-trivial closed normal subgroups of $G$. By Theorem~\ref{thm:BurgerMozes}, however, the group $G^{(\infty)}$ is topologically simple and compactly generated. Thus we have arrived at a contradiction, since a compactly generated totally disconnected locally compact group that is topologically simple cannot have any soluble open subgroup by \cite[Theorem~2.2]{Willis07}.
\end{proof}

\begin{remark}\label{rem:Ito}
Conjecture~\ref{conj:KW} is known to hold for products of abelian groups, which are necessarily metabelian by a classical result of Ito \cite{Ito1955}. The proof of Corollary~\ref{cor:Hua:NilpotentRootGps} therefore shows that Conjecture~\ref{conj:VirtCyclic:2} holds if the root groups are abelian. 
\end{remark}

\begin{remark}
The tight relations between the above conjectures provide strong evidence that Conjecture~\ref{conj:VirtCyclic:1} holds. Indeed, assume that this is not the case and let $G \leq \Aut(T)$ be a counterexample. By Lemma~\ref{lem:EquivConj} the group $G$ is then boundary-Moufang. If root groups are nilpotent, then Corollary~\ref{cor:Hua:NilpotentRootGps} implies that Kegel's conjecture~\ref{conj:KW} fails. Otherwise we have an example of a proper Moufang set with non-nilpotent root groups, which contradicts one of the main conjectures on Moufang sets. 
\end{remark}

\subsection{Reducing to the pro-$p$ case}

A theorem by Elizabeth Pennington~\cite{Pennington} ensures that if a finite group $G$ is a product $G= A.B$ of two nilpotent groups of respective classes $a$ and $b$, then the $(a+b)^{\text{th}}$ derived group $G^{(a+b)}$ is nilpotent. This implies in particular that Kegel's Conjecture~\ref{conj:KW} can be reduced to the special case that $G$ is a $p$-group. 

Here we shall use Pennington's result to establish the following restriction on a possible counterexample to Conjecture~\ref{conj:VirtCyclic:2}. 

\begin{prop}\label{prop:VirtCycl}
Let $T$ be a locally finite tree and $G \leq \Aut(T)$ be boundary-Moufang with nilpotent root groups. Assume that the stabiliser  $G_{\xi, \xi'}$ of a pair of distinct boundary points $\xi, \xi' \in \bd T$ is virtually cyclic. 

Then the little projective group $G^+$ is open of finite index in $G$, and is abstractly simple. Moreover there is a prime $p$ such that the pointwise stabiliser in $G^+$  of each ball of radius~$2$ is a pro-$p$ group. In particular, all compact open subgroups of $G$ are virtually pro-$p$ groups. 
\end{prop}

\begin{proof}
By Lemma~\ref{lem:Hua:NilpotentRootGps}, we find a compact open subgroup  $V < G$ which is a product of two nilpotent groups $A$ and $B$ of classes $a$ and $b$, respectively. Moreover $A$ and $B$ are compact subgroups of root groups. By~\cite{Pennington}, it follows that $\overline{V^{(a+b)}}$ is a pro-nilpotent normal subgroup of $V$. 

Notice that by Lemma~\ref{lem:Hua:NilpotentRootGps} the little projective group $G^+$ is open in $G$. Since the root groups are nilpotent, it follows from Proposition~\ref{prop:LittleProjSimple} that $G^+$ is abstractly simple. Moreover $G^+$ has finite index in $G$ by Proposition~\ref{prop:LittleProjMonolith}. 

Since $[G:G^+]$ is finite and $G^+$ is itself boundary-Moufang, the compact open subgroup $V < G$ provided by Lemma~\ref{lem:Hua:NilpotentRootGps} can in fact be chosen in $G^+$. Since $G^+$ has no compact open soluble subgroup by  \cite[Theorem~2.2]{Willis07}, it follows that $V^{a+b}$ is non-trivial, and hence there is some prime $p$ such that the maximal normal pro-$p$ subgroup $O_p(V)$ is non-trivial. By \cite[Theorem~4.8]{BEW}, a compactly generated simple locally compact group cannot have any non-trivial finite subgroup with an open normaliser. Thus $O_p(V)$ is infinite. It follows that $O_p(V\cap W)$ is non-trivial for each compact open subgroup $W$ because $[V : V \cap W]$ is finite. Since $V \cap W$ contains an open subgroup $W'$ which is normal in $W$, we deduce that $O_p(W') \leq O_p(W)$ is non-trivial. Thus every compact open subgroup $W$ of $G^+$  has $O_p(W)$ non-trivial.

This applies in particular to all vertex stabilisers. The desired conclusion now follows from \cite[Proposition~2.1.2(2)]{BurgerMozes:trees1}, which ensures that if $x \in T$ is such that $O_p(G^+_x)$ is non-trivial, then the pointwise stabiliser of the ball of radius~$2$ around any neighbour of $x$ is a pro-$p$ group (the statement of Proposition~2.1.2(2) in~\cite{BurgerMozes:trees1} is formally slightly weaker than this, but the proof indeed yields the above assertion).  
Notice that {\em loc.\@ cit.\@} assumes the group $G^+$ to be locally primitive, which is indeed the case because it acts
doubly transitively on the boundary $\bd T$; see also \cite[Lemma 3.1.1]{BurgerMozes:trees1}.
%
\end{proof}

\begin{remark}
In view of Lemma~\ref{lem:RootGroups:contraction}, root groups are subjected to  structural restrictions established by Gl\"ockner--Willis \cite{GloecknerWillis} (see  Theorem~\ref{thm:GlocknerWillis}). When $G$ is locally pro-$p$ as in the conclusions of Proposition~\ref{prop:VirtCycl}, it follows that root groups are direct products of a $p$-adic analytic nilpotent group and a torsion $p$-group of bounded exponent. We shall see in Corollary~\ref{cor:padic:Hua} below that, under the hypotheses of Proposition~\ref{prop:VirtCycl},  the torsion part must  be non-trivial. 
\end{remark}

\subsection{Fixed trees and sub-Moufang sets}\label{sec:subMoufang}

A $\tau$-invariant subgroup of $U$ in a Moufang set $\MM(U, \tau)$ is called a \textbf{root subgroup} and gives rise to a \textbf{sub-Moufang set}, see  \cite[Section~6.2]{DeMedtsSegev:course}. In particular, it is shown in  Lemma~6.2.3 from \emph{loc.~cit.} that for any element $h $ in the Hua subgroup $H$, the centraliser  $\centra_U(h)$ is a root subgroup as soon as it is non-trivial. The argument given shows in fact that this holds for the centraliser $\centra_U(K)$ of any subgroup $K \leq H$. 

In the present context, we will need the following version of the latter fact. 

\begin{lem}\label{lem:SubMoufangSet}
Let $G \leq \Aut(T)$ be boundary-Moufang and let $K \leq G$ be a compact subgroup fixing two distinct points $0, \infty \in \bd T$. 

If $K$ fixes a third point in $\bd T$, then the centraliser $G^K = \centra_G(K)$ preserves a unique minimal sub-tree $X$ contained in   the fixed tree $T^K$, and the $G^K$-action on $X$ is boundary-Moufang. Moreover $T^K$ is contained in a bounded neighbourhood of $X$. 
\end{lem}

\begin{proof}
We first observe that $T^K$ is non-empty since $K$ is compact.
Since $K$ fixes two distinct points of the boundary $\bd T$,
this implies in particular that $(\bd T)^K = \bd (T^K)$, so that the notation $\bd T^K$ is unambiguous. 

Let $\xi \in \bd T^K$ be any point distinct from $0, \infty$. Let $u \in U_\infty$ be the unique element such that $\xi = u.0$. Then  for any $k \in K$, the commutator $[k, u]= k u k\inv u\inv$ fixes $\xi$. Moreover $[k, u]$ belongs to $U_\infty$ since $K$ normalises $U_\infty$. Thus $[k, u]= 1$ and $u \in \centra_{U_\infty}(K)$.

This shows that the group $U_\infty^K =   \centra_{U_\infty}(K)$ acts sharply transitively on $\bd T^K \setminus \{\infty\}$. Similarly one shows that $U_0^K =   \centra_{U_0}(K)$ acts sharply transitively on $\bd T^K \setminus \{0\}$. This implies that  the pair $(\bd T^K, G^K)$ is a Moufang set. Let $X$ denote the unique minimal $G^K$-invariant sub-tree of $T^K$. Then $\bd X \subseteq \bd T^K$, and since $G^K$   is transitive on $\bd T^K$, it follows that $\bd X = \bd T^K$. By construction $G^K$ and $U_0^K$ are closed. Moreover, the image of $G^K$ in $\Aut(X)$ is closed since $G^K$ acts properly on $X$. It follows that the $G^K$-action on the tree $X$ satisfies the  boundary-Moufang condition. In particular the $G^K$-action on $X$ is cocompact by \cite[Lemma 3.1.1]{BurgerMozes:trees1}. 

It remains to see that $X$ is contained in a bounded neighbourhood of $X$. If this were not the case, then we would find a sequence of vertices $(v_n)$ in $T^K$ whose distance to $X$ tends to infinity. Let $x_n \in X$ denote the projection of $v_n$ to $X$. Since $G^K$ acts cocompactly on $X$, we can find $g_n \in G^K$ such that $x'_n = g_n.x_n$ is bounded. Since $G^K$ preserves $T^K$, we have $v'_n = g_n.v_n \in T^K$. Upon extracting we may assume that $(x'_n)$ is constant. It follows that the sequence of geodesic segments $[x'_0, v'_n]$ subconverges to a geodesic ray, contained in $T^K$, whose endpoint does not belong to $\bd X$ since $x'_0$ is the projection of $v'_n$ to $X$ for all $n$. This contradicts the fact that $\bd X = \bd T^K$.
\end{proof}

It is a general (and easy) fact that the centraliser $\centra_H(U)$ of the root group $U$ in the Hua subgroup $H$ is trivial in any Moufang set $\MM(U, \tau)$. Specialising Lemma~\ref{lem:SubMoufangSet} to the case of compact open subgroups of the Hua subgroup, we get the following related result. 

\begin{lem}\label{lem:CentraliserHua}
Let $G \leq \Aut(T)$ be boundary-Moufang, let $0, \infty \in \bd T$ be distinct boundary point and let $K$ be a compact open subgroup of the Hua subgroup  $G_{0, \infty}$. 

If $\centra_{U_\infty}(K)$ is non-trivial, then $G^K = \centra_G(K)$ preserves a unique minimal sub-tree $X$ contained in   the fixed tree $T^K$, and the $G^K$-action on $X$ is boundary-Moufang with root group $\centra_{U_\infty}(K)$   and with virtually cyclic Hua subgroup.
In particular, if the root groups are nilpotent and if Conjecture~\ref{conj:KW} holds, then  $\centra_{U_\infty}(K)=1$.
\end{lem}

\begin{proof}
If $\centra_{U_\infty}(K)=1$  is not trivial, then $K$ fixes at least three points in $\bd T$ and Lemma~\ref{lem:SubMoufangSet} shows that $G^K = \centra_G(K)$ acts as a boundary-Moufang group on $T^K$.  Since $K$ acts trivially on $T^K$, the boundary-Moufang group $G^K /K \cap G^K$ acts on $T^K$, and by Lemma~\ref{lem:Hua}, the Hua subgroup is virtually cyclic. In the case that the root groups are nilpotent and Conjecture~\ref{conj:KW} holds, this contradicts Corollary~\ref{cor:Hua:NilpotentRootGps}. 
\end{proof}

\section{Abelian root groups and $\SL_2$}

\subsection{General Moufang sets with abelian root groups}

A major conjecture in the theory of Moufang sets asserts that every proper Moufang set with abelian root groups comes from a quadratic Jordan division algebra, see \cite[\S7.6]{DeMedtsSegev:course}. A number of results have already been obtained in this direction, some of which are collected in the following.

\begin{thm}\label{thm:recap:MoufangSet}
Let $\MM(U, \tau)$ be a proper Moufang set with abelian root groups (written additively), and let $H \leq G^+$ be its Hua subgroup.
Then the following assertions hold. 
\begin{enumerate}[\rm (i)]
\item  $\MM(U, \tau)$ is special, \emph{i.e.} we have $\tau(-u) = -\tau(u)$ for all $u \in U$.  

\item $U$ is either torsion-free and divisible, or of exponent $p$ for some prime $p$. 

\item If $U$ is not of exponent $2$, then $U$ does not have any non-trivial $H$-invariant subgroups. 
\end{enumerate}
\end{thm}

\begin{proof}
(i) is the main result from \cite{Segev09} while (iii) is the main result from \cite{SegevWeiss08}.
For (ii) we refer to \cite[I.5.2(a)]{Timmesfeld}; see also \cite[Proposition 7.2.2(5)]{DeMedtsSegev:course}.
\end{proof}

\subsection{Abelian Hua subgroup}

In the special case that not only the root groups are abelian, but also the Hua subgroup, one has a complete classification of all possible Moufang sets. Since all such Moufang sets are special, as mentioned above, this classification follows from  \cite[Theorem~6.1]{DeMedtsWeiss06} and the main result from \cite{Grueninger}. This classification asserts that every such Moufang set is the projective line over a field, except in characteristic~$2$ where well-understood exceptions arise from imperfect fields $K$ such that $[K:K^2] > 2$, where $K^2$ is the subfield consisting of the squares; see \cite{Grueninger}. 

It turns out that for Moufang sets with abelian root groups, the condition that the Hua group is abelian is equivalent to the condition that the point-stabilisers are metabelian: 

\begin{prop}\label{prop:Habelian}
	Let $\MM(U, \tau)$ be a Moufang set with $U$ abelian.
	Then the Hua subgroup $H$ is abelian if and only if the point stabilizer $G_\infty$ is metabelian.
\end{prop}
\begin{proof}
	If $\MM(U, \tau)$ is not proper, then $H = 1$ and $G_\infty = U_\infty$, in which case the statement is trivially true.
	So assume that $\MM(U, \tau)$ is proper;
	by the main result of \cite{Segev09}, the Moufang set is special.
	If $H$ is abelian, then $G_\infty = U_\infty \rtimes H$ is indeed metabelian.

	Assume now that $G_\infty$ is metabelian; we will show that $H$ is abelian.
	If $|U| \leq 3$, then $\MM(U, \tau)$ is the projective line over the field with $2$ or $3$ elements,
	so in this case $H$ is indeed abelian.
	So assume that $|U| > 3$; then we can invoke \cite[Theorem 1.12]{DeMedtsSegevTent}, and hence $[U_\infty, H] = U_\infty$;
	in particular, $U_\infty \leq [G_\infty, G_\infty]$.
	Since $[G_\infty, G_\infty]$ is abelian, this implies that $[G_\infty, G_\infty] \leq \centra_{G_{\infty}}(U_\infty)$.
	On the other and, $G_\infty = U_\infty \rtimes H$, but no element of $H$ centralizes $U_\infty$;
	it follows that $\centra_{G_{\infty}}(U_\infty) \leq U_\infty$.
	We conclude that $U_\infty = [G_\infty, G_\infty]$.
	Therefore $H \cong G_\infty / U_\infty \cong G_\infty / [G_\infty, G_\infty]$ is abelian.
\end{proof}

In our setting, we shall show that all the fields which appear naturally are local fields; in particular it will follow that the aforementioned characteristic~$2$ exceptions do not arise. 

\begin{lem}\label{lem:AbelianHua}
Let $G \leq \Aut(T)$ be boundary-Moufang and assume that for some pair of distinct points $0, \infty \in \bd T$, the root group $U_\infty$ and the stabiliser $G_{0, \infty}$ are both abelian. 

Then the little projective group $G^+$ is closed, cocompact in $G$ and isomorphic as a topological group to $\PSL_2\big(k)$, where $k$ is a non-Archimedean local field. Moreover $G$ embeds as a closed subgroup in $\PGL_2(k)$ and $T$ is equivariantly isomorphic to the Bruhat--Tits tree of $\PGL_2\big(k \big)$.
\end{lem}

\begin{proof}
By Theorem~\ref{thm:recap:MoufangSet}, the (abstract) group $U$ (written additively) is canonically endowed with the structure of a vector space over $F$, where $F=\QQ$ or $\FF_p$ for some prime $p$. 

\medskip
Assume first that $F \neq \FF_2$ and let $e \in U^*$ be some fixed non-zero element. Let $\tau = \mu_e$, and let
$h \colon U \to G_{0, \infty} \colon u \mapsto h_u$ be the map defined by $h_u = \mu_u \tau$ if $u \neq 0$, and $h_0 = 0$.
Now \cite[Theorem~6.1]{DeMedtsWeiss06} ensures that the map 
\[ U \times U \to U \colon (x,y) \to \frac 1 2 (h_{e+x} - h_e - h_x)(y) \]
is a multiplication that turns $U$ into a field $k$. Moreover the little projective group $G^+$ is (abstractly) isomorphic to $\PSL_2(k)$ and the Moufang set $\partial T$ is equivariantly isomorphic to the projective line over $k$. Recalling now that $U$ is a locally compact additive group, and that the map $h \colon u \mapsto h_u$ is continuous by Proposition~\ref{prop:mu-map:continuous}, it follows that $k$ is in fact a non-discrete locally compact topological field. Since $k$ is totally disconnected, it follows  that $k$ is non-Archimedean. 

Since every continuous homomorphism from $\PSL_2(k)$ to a locally compact group $G$ is proper (see \cite[Lemma~5.3]{BM96}), it follows that $G^+$ is closed in $G$. Since $G^+$ acts $2$-transitively on $\bd T$, it is edge-transitive on $T$ and hence cocompact in $G$. 
Since $\Aut(\PSL_2(k)) = \mathrm P\Gamma \mathrm L_2(k)$ and since $\centra_{\Aut(T)}(G^+)=1$, it follows that $G$ can be identified with a subgroup of $\mathrm P \Gamma \mathrm L_2(k)$ . Since the stabiliser of a pair is abelian in $G$, no non-trivial field automorphism is involved in $G$ and, hence, we have $G \leq \PGL_2(k) \leq \Aut(T)$. The fact that $T$ is equivariantly isomorphic to the Bruhat--Tits tree of $\PGL_2(k)$ is now an easy consequence of elementary Bass--Serre theory. Since $G$ is closed in $\Aut(T)$, it is closed in $\PGL_2(k)$.  The desired result follows in this case. 

\medskip
Assume now that $F = \FF_2$. Let $k \subset \End(U)$ be the subring generated by the Hua subgroup $H = G^+ \cap G_{0, \infty}$. By Part (1) of the main theorem from \cite{Grueninger}, the ring $k$ is a field. Moreover Lemma~\ref{lem:PrimeExponent} ensures that $k$ is a finite-dimensional algebra over the local field $\FF_2(\!(t) \!)$. It follows that $k$ is itself a local field, and must  therefore be equal to $\FF_q(\!(t)\!) $, where $q=2^d$ for some $d>0$. 

In particular, the field $\sqrt k$ consisting of all the elements of an algebraic closure of $k$ whose square belongs to $k$, is an extension of $k$ of degree $2$: it coincides with $\FF_q(\! ( \sqrt t) \! )$. Therefore, Part (3) of the main theorem from \cite{Grueninger} ensures that the little projective group $G^+$ is isomorphic to $\PSL_2(k)$ and the Moufang set $\partial T$ is equivariantly isomorphic to the projective line over $k$. Since the $\mu$-maps are continuous by Proposition~\ref{prop:mu-map:continuous}, the isomorphism between $G^+$ and $\PSL_2(k)$ must be a homeomorphism. We can now conclude the proof in the same way as in the characteristic~$\neq 2$ case.
\end{proof}

We now come to the proof of Theorem~\ref{thm:MetabelianStab}.

\begin{proof}[Proof of Theorem~\ref{thm:MetabelianStab}]
By Proposition~\ref{prop:2trans}, the group $G$ is $2$-transitive on $\bd T$. In particular it is edge-transitive on $T$ and hence contains some hyperbolic element $\alpha$. Let $\infty$ be the repelling fixed point of $\alpha$. By Lemma~\ref{lem:parab}, the parabolic subgroup $P_\alpha$ coincides with the stabiliser $G_\infty$. Since $G$ is transitive on $\bd T$, the group $G_\infty = P_\alpha$ is conjugate to $G_\xi$, which is metabelian by hypothesis. Therefore $P_\alpha$ is metabelian. 

\medskip
We next claim that the contraction group $U_\alpha$ is abelian. Indeed, let $u \in U_\alpha$. Then for all $n>0$ we have $\alpha^n u\inv \alpha^{-n} u \in [P_\alpha, P_\alpha]$. Therefore we have
\[
u = \lim_{n\to \infty} \alpha^n u\inv \alpha^{-n} u \in \overline{[P_\alpha, P_\alpha]}.
\]
This proves that $U_\alpha \leq \overline{[P_\alpha, P_\alpha]}$, which is abelian since $P_\alpha$ is metabelian. This confirms the claim. 

\medskip
By Corollary~\ref{cor:AbelianContraction}, we deduce that the contraction group $U_\alpha$ is closed. Theorem~\ref{thm:ClosedContraction} thus implies that $G$ is boundary-Moufang and $U_\alpha$ is a root group.
%
%
%
By Proposition~\ref{prop:Habelian}, the Hua subgroup $H = G_{0, \infty}$ is abelian.
The desired conclusion now follows from Lemma~\ref{lem:AbelianHua}.
\end{proof}

\section{Torsion-free root groups}

The main result in this section is a complete classification of boundary-Moufang groups with torsion-free root groups; all of them come from rank one semisimple algebraic groups over $p$-adic fields.

\begin{thm}\label{thm:TorsionFree}
Let $G \leq \Aut(T)$ be boundary-Moufang. Assume that the root groups are torsion-free. 

Then the little projective group $G^+$ is open of finite index in $G$, and there is a prime $p$, a $p$-adic field $k$ and a semisimple algebraic $k$-group $\mathbf G$ of $k$-rank one such that $G^+ \cong \mathbf G(k)$ and $T$ is equivariantly isometric (up to scaling) to the Bruhat--Tits tree of $\mathbf G(k)$. 
\end{thm}

Given Theorem~\ref{thm:ClosedContraction} and Proposition~\ref{prop:LittleProjMonolith}, the latter result immediately implies the following reformulation of Theorem~\ref{thm:padic}. 

\begin{cor}\label{cor:padic:reformulation}
Let $G \leq \Aut(T)$ be closed and boundary-transitive. Assume that $G$ contains a hyperbolic element whose associated  contraction group  is  closed and torsion-free. 

Then the monolith $G^{(\infty)}$ is open of finite index in $G$, and there is a prime $p$, a $p$-adic field $k$ and a semisimple algebraic $k$-group $\mathbf G$ of $k$-rank one such that $G^{(\infty)} \cong \mathbf G(k)$ and $T$ is equivariantly isometric (up to scaling) to the Bruhat--Tits tree of $\mathbf G(k)$. 
\end{cor}

\subsection{Local recognition}

The strategy of proof of Theorem~\ref{thm:TorsionFree} is based on the following fact, showing that it suffices to establish  the analyticity locally. 

\begin{prop}\label{prop:padic:local}
Let $G \leq \Aut(T)$ be closed, non-compact and boundary-transitive. 

If some compact open subgroup $U < G$ is $p$-adic analytic for some prime $p$, then the monolith $G^{(\infty)}$ is open of finite index in $G$, and there is a  $p$-adic field $k$ and a  semisimple linear algebraic $k$-group $\mathbf G$ of $k$-rank one such that $G^{(\infty)} \cong \mathbf G(k)$ and $T$ is equivariantly isometric (up to scaling) to the Bruhat--Tits tree of $\mathbf G(k)$. 
\end{prop}
 
\begin{proof}
By Proposition~\ref{prop:LittleProjMonolith} the monolith $G^{(\infty)}$ is topologically simple. If $U$ is $p$-adic analytic, then   so is $G$ by \cite[Theorem~8.32]{DDMS}. By the classification of simple $p$-adic Lie algebras (see \cite{Weisfeiler}), it follows that every topologically simple $p$-adic analytic group is isomorphic to the group  $\mathbf G(k)$ of $k$-rational points of an adjoint semisimple algebraic $\mathbf G$ defined over a finite extension $k$ of $\QQ_{p}$, and of $k$-rank one. Thus $G^{(\infty)} \cong \mathbf G(k)$.

Remark that $\mathbf G(k)$ has finite outer-automorphism group: this follows from \cite[Corollary~8.13]{BorelTits73} and the fact that $k$ has a finite automorphism group. Moreover we have  $\centra_G( G^{(\infty)}) = 1$ because this centraliser must act trivially on $\bd T$, and hence on $T$. It follows that $G^{(\infty)}$ is open in $G$, hence of finite index by Proposition~\ref{prop:LittleProjMonolith}.
Bass--Serre theory implies that any infinite locally finite tree on which $\mathbf G(k)$ acts continuously and edge-transitively is equivariantly isomorphic to the Bruhat--Tits tree $\mathbf G(k)$, and this concludes the proof.
\end{proof}

This has the following noteworthy consequence regarding the structure of the Hua subgroup of a boundary-Moufang group. 
In fact, it is an interesting open problem whether there exists an (abstract) Moufang set with vitually cyclic infinite Hua subgroup.
\begin{cor}\label{cor:padic:Hua}
Let $G \leq \Aut(T)$ be boundary-Moufang with torsion-free root groups. 

Then the Hua subgroup is not virtually cyclic. 
\end{cor}

\begin{proof}
Assume by contradiction that the Hua subgroup is virtually cyclic.
By Lemma~\ref{lem:RootGroups:contraction}, the root group $U$ is a contraction group. Since $U$ is torsion-free by hypothesis, it follows from Theorem~\ref{thm:GlocknerWillis} that root groups are nilpotent and decompose as direct products of $p$-adic analytic groups over various primes. By Proposition~\ref{prop:VirtCycl} there is in fact only a single prime involved, say $p$. By Proposition~\ref{prop:pro-p}, this implies that compact open subgroups in $G$ are virtually $p$-adic analytic. Therefore we can invoke Proposition~\ref{prop:padic:local}, ensuring that $G$ is a semisimple $p$-adic algebraic group acting on its Bruhat--Tits tree. A contradiction arises since the Hua subgroup coincides with the Levi subgroup of a parabolic, and is thus not virtually cyclic. 
\end{proof}

We shall also need the following property of Moufang sets associated to rank one semisimple algebraic groups over $p$-adic fields. 

\begin{lem}\label{lem:padicMoufang}
Let $\MM(U, \tau)$ be the standard Moufang set  associated to a rank one semisimple algebraic group over some $p$-adic field. 

Then there is a subgroup $H'$ of the Hua subgroup such that $\centra(U) = \centra_U(H')$. 
\end{lem}

\begin{proof}
If $U$ is abelian, then this is obvious; so assume that $U$ is non-abelian.
By the classification of semisimple algebraic groups \cite{Tits:Boulder}, there is only one remaining class of rank one groups over
$p$-adic fields, namely the rank one groups corresponding to skew-hermitian forms of Witt index one.
The corresponding Moufang sets are precisely the ones that we described in Example~\ref{ex:skewherm}.

So let $\MM(U, \tau)$ be as in Example~\ref{ex:skewherm}, and observe that $\tau^2$ maps an element $(a,s) \in U$ to $(-a,s)$.
Moreover, $\tau^2$ is an element of $G^+$ fixing $0$ and $\infty$, and hence it belongs to the Hua subgroup $H$.
Since the skew-hermitian form $h$ is non-degenerate, the center $\centra(U)$ coincides with $\{ (0,s) \mid s \in \Fix_K(\sigma) \}$,
and hence $\centra(U) = \centra_U(\tau^2)$, proving the claim.
\end{proof}

\subsection{Classification in characteristic 0}

\begin{proof}[Proof of Theorem~\ref{thm:TorsionFree}]
Let $0, \infty \in \bd T$ be distinct boundary points and set $U = U_\infty$. By Lemma~\ref{lem:RootGroups:contraction}, the root group $U$ is a contraction group. Since $U$ is torsion-free by hypothesis, it follows from Theorem~\ref{thm:GlocknerWillis} that there are primes $p_1, \dots, p_n$ and for each $i$ some infinite closed $p_i$-adic analytic nilpotent subgroup $ U_i \leq U$ such that $U \cong U_1 \times \dots \times U_n$. In particular root groups are nilpotent. 

Our goal is to show that $n =1$ or, equivalently, that $U$ is $p$-adic analytic for some fixed prime $p$. 

Let $H = G_{0, \infty}$ be the Hua subgroup and let $H_c \leq H$ be its unique maximal compact open normal subgroup (see Lemma~\ref{lem:Hua}). Let 
\[
\varphi \colon H \to \Aut(U) = \Aut(U_1) \times \dots \times \Aut(U_n)
\]
be the continuous homomorphism induced by the conjugation action of $H$ on $U$, and let $\varphi_i \colon H \to \Aut(U_i)$ be the induced $H$-action on $U_i$.  Since $\Aut(U_i)$ acts faithfully on the $\QQ_{p_i}$-Lie algebra of $U_i$, it follows that any compact subgroup of $\Aut(U_i)$ is isomorphic to a closed subgroup of $\GL_d(\QQ_{p_i})$ for some $d$. Since any closed subgroup of a $p_i$-adic analytic group is itself analytic (see \cite[Sect.~V.9]{Serre65}), it follows that $\Aut(U_i)$ is $p_i$-adic analytic and, in particular, locally pro-$p_i$. In particular $\Aut(U)$ is locally pro-nilpotent and, hence, the compact group $\varphi(H_c)$ is virtually pro-nilpotent.
Notice that $\varphi$ is injective because the conjugation action of $H$ on $U$ is equivalent with the action of $H$ on the
elements of $\bd T \setminus \{ \infty \}$, which is faithful by definition.
It follows that $H_c$ is virtually pro-nilpotent as well.
Therefore $H_c$ has an open subgroup $K$ isomorphic to $K \cong H_1 \times \dots \times H_n$, where $H_i$ is a pro-$p_i$ subgroup of $H_c$. Moreover   $\varphi_i(H_i)$ is open in $\varphi_i(H_c)$, and $\Ker(\varphi_i) \cap H_i = 1$ since $\varphi$ is injective; and $\varphi_j(H_i) =1 $ for all $j \neq i$. 


\medskip
For each $i \in \{1, \dots, n\}$, set $K_i = \la H_j : j \neq i\ra$. 
We claim that $\centra_U(K_i) = U_i$. 

Clearly $U_i \leq \centra_U(K_i)$. If the inclusion were proper, then $\centra_U(K_i)$ would contain some non-trivial element $u \in \la U_j : j \neq i\ra$. Since $u$ also commutes with $H_i$, it would follow that $u \in \centra_U(K)$. By Lemma~\ref{lem:CentraliserHua}, this implies that $\centra_G(K)/K$ is a boundary-Moufang group with torsion-free root group and virtually cyclic Hua subgroup. This is impossible by Corollary~\ref{cor:padic:Hua}. Thus $\centra_U(K_i) = U_i$ as claimed.

\medskip
By Lemma~\ref{lem:SubMoufangSet}, the group $\centra_G(K_i)$ acts as a boundary-Moufang set on $T^{K_i}$ with $\centra_U(K_i) = U_i$ as a root group. Let $\widehat K_i$ be the kernel of the $\centra_G(K_i)$-action on $T^{K_i}$. Since $0, \infty \in \bd T^{K_i}$, it follows $\widehat K_i$ is a compact subgroup of $H$ which contains $K_i$. 

From Proposition~\ref{prop:pro-p}, we deduce that $\centra_G(K_i)/\widehat K_i$ is locally $p_i$-adic analytic. In view  of Propositions~\ref{prop:LittleProjMonolith} and~\ref{prop:padic:local}, we infer that there is a  $p_i$-adic field $k_i$ and a  semisimple linear algebraic $k_i$-group $\mathbf G_i$ of $k_i$-rank one such that $G^{(\infty)} = \overline{G^+} \cong \mathbf G_i(k_i)$. Thus the Moufang set $\bd T^{K_i}$ is an algebraic Moufang set over the $p_i$-adic field $k_i$. From Lemma~\ref{lem:padicMoufang}, it follows that there is some subgroup $H'_i \leq H_i$ such that $\centra(U_i) = \centra_{U_i}(H'_i)$. 

We finally set $K' =  H'_1 \times \dots \times H'_n$. Thus we have $\centra_{U}(K') = \centra(U)$. By Lemma~\ref{lem:SubMoufangSet}, it follows that $\centra(U)$ is the root group of some sub-Moufang set. From Theorem~\ref{thm:recap:MoufangSet}(iii), it follows that $\centra(U) = \centra(U_1) \times \dots \times \centra(U_n)$ has no non-trivial subgroup invariant under the Hua subgroup of that Moufang set. Since $\centra(U_i) \neq 1$ for all $i$, we infer that $n=1$, as desired. 

We set $\mathbf G= \mathbf G_1$ and $k= k_1$. By \cite{Platonov}, the group $\mathbf G(k)$ is abstractly simple. Therefore we must have $G^+ = \overline{G^+}$, and $G^+$ is closed, as desired. 
\end{proof}


\appendix

\section{Multiply transitive groups with soluble stabilisers}


In this closing section, we point out that projective groups $\PGL_2(k)$ over commutative fields can be characterised among \emph{abstract} triply transitive permutation groups. Various such characterisations were established by  J.~Tits in his PhD thesis~\cite{Tits_thesis}. In the following result, we record only one of them, as it is closely related to Theorem~\ref{thm:MetabelianStab} from the introduction. Notice however that Theorem~\ref{thm:MetabelianStab} cannot be derived from those abstract characterisations, since it also applies to $\PSL_2(k)$, which is not triply transitive in general.


\begin{thm}\label{thm:n-ply:transitive}
Let $n >0 $ and $G \leq \Sym(X)$ be a permutation group.
If $G$ is \mbox{$(n+1)$}\nobreakdash-\hspace{0pt}transitive and the point-stabilisers are soluble of length at most $n$, then one of the following holds. 

\begin{enumerate}[\rm (i)]
\item $n=1$ and $G$ is the group of linear maps over a field $k$, and $X = \mathbb A^1(k)$.

\item $n=2$ and there is a field $k$ such that $ G = \PGL_2(k)$, and $X = \mathbb P^1(k)$. 

\item $n=3$ and $G = \Sym(4)$ or $G = \Sym(5)$. 

\item $n=4$ and $G = \Sym(5)$.  
\end{enumerate}

\end{thm}

The case $n=2$ can easily be derived from a result of J.~Tits~\cite{Tits_thesis}. The case $n=1$ was established much more recently in \cite{KKP_affine} and \cite{Mazurov} (independently of each other).

\begin{proof}
 \cite{KKP_affine} or \cite{Mazurov} cover the case $n=1$; we assume henceforth that $n>1$. 
 
Let $x \in G$. By assumption $G_x$ is  $n$-transitive on $X \setminus \{x\}$. The last term $U_x$ of the derived series of $G_x$ is abelian. As an abelian normal subgroup of an $n$-transitive group (with $n>1$), the group $U_x$ must therefore act regularly on $X\setminus \{x\}$. This proves that $(X, (U_x)_{x \in X})$ is a Moufang set with abelian root groups. Moreover, given a pair $x, y \in X$ with $x \neq y$, we have  $U_x  \cap G_{x, y} = 1$ and, hence, the group $G_{x, y}$ is soluble of derived length at most $n-1$. We can thus apply induction on $n$.

Since a group containing $\PSL_2(k)$ as a subgroup can be soluble only if $k = \FF_2$ or $\FF_3$, the theorem will follow once we finish the proof in the case $n=2$. 

In case $n=2$, we have seen that the double stabiliser $G_{x, y}$ is abelian. Since any abelian transitive group acts freely, it follows that $G$ is sharply $3$-transitive. The desired conclusion then follows from Th.~VI on p.~45 in \cite{Tits_thesis}, which asserts that a sharply $3$-transitive group in which double stabilisers are abelian, is necessarily a projective group $\PGL_2(k)$ over a field $k$, acting on its projective line.
%
%
%
\end{proof}

\bibliographystyle{amsalpha}

\providecommand{\bysame}{\leavevmode\hbox to3em{\hrulefill}\thinspace}
\providecommand{\MR}{\relax\ifhmode\unskip\space\fi MR }
\providecommand{\MRhref}[2]{%
  \href{http://www.ams.org/mathscinet-getitem?mr=#1}{#2}
}
\providecommand{\href}[2]{#2}

\end{document}